\documentclass[12pt]{amsart} 

\usepackage{color}

\usepackage{mathrsfs}
\usepackage{amssymb,amsmath,latexsym}
\usepackage{graphicx}  

\newcommand{\ekv}[2]{\begin{equation}\label{#1}#2\end{equation}}

\newcommand{\be}{\begin{equation}}
\newcommand{\ee}{\end{equation}}

\theoremstyle{plain}

\newtheorem{thm}{Theorem}
\newtheorem{prop}{Proposition}[section]

\newtheorem{lem}[prop]{Lemma}

\theoremstyle{definition}

\numberwithin{equation}{section}

\def\bbbone{{\mathchoice {1\mskip-4mu {\rm{l}}} {1\mskip-4mu {\rm{l}}}
{ 1\mskip-4.5mu {\rm{l}}} { 1\mskip-5mu {\rm{l}}}}}

\def\squarebox#1{\hbox to #1{\hfill\vbox to #1{\vfill}}}

\newcommand{\tA}{\widetilde A}

\newcommand{\tkappa}{\widetilde \kappa}

\newcommand{\tD}{\widetilde D}

\newcommand{\tF}{\widetilde F}

\newcommand{\wtG}{\widetilde G}

\newcommand{\tM}{\widetilde M}
\newcommand{\tOmega}{\widetilde \Omega}

\newcommand{\tpsi}{\widetilde \psi}

\newcommand{\tT}{\widetilde T}

\newcommand{\wtu}{\widetilde u}
\newcommand{\wtv}{\widetilde v}
\newcommand{\tV}{{\widetilde V}}

\newcommand{\wSi}{{\widetilde \Sigma}}
\newcommand{\tSigma}{{\widetilde \Sigma}}

\newcommand{\whD}{{\widehat D}}
\newcommand{\whK}{{\widehat K}}

\newcommand{\whT}{{\widehat T}}

\newcommand{\eps}{\epsilon}

\newcommand{\bSigma}{\boldsymbol{\Sigma}}

\newcommand{\RR}{{\mathbb R}}

\newcommand{\CC}{{\mathbb C}}

\newcommand{\NN}{{\mathbb N}}
\newcommand{\IR}{{\mathbb R}}

\def\t2{{\mathbb T}^2}

\newcommand{\CI}{{\mathcal C}^\infty }
\newcommand{\CIc}{{\mathcal C}^\infty_{\rm{c}} }

\newcommand{\Oo}{{\mathcal O}} 

\newcommand{\cD}{{\mathcal D}}
\newcommand{\cE}{{\mathcal E}}
 
\newcommand{\TT}{{\mathcal T}}

\newcommand{\cH}{{\mathcal H}}
\newcommand{\cA}{{\mathcal A}}

\newcommand{\cM}{{\mathcal M}}

\newcommand{\cO}{{\mathcal O}}
\newcommand{\cP}{{\mathcal P}}
\newcommand{\cR}{{\mathcal R}}
\newcommand{\cS}{{\mathscr S}}

\newcommand{\defeq}{\stackrel{\rm{def}}{=}}

\newcommand{\rank}{\operatorname{rank}}
\newcommand{\supp}{\operatorname{supp}}

\newcommand{\WF}{\operatorname{WF}}
\newcommand{\WFh}{\operatorname{WF}_h}

\newcommand{\tr}{\operatorname{tr}}
\newcommand{\ind}{\operatorname{ind}}

\newcommand{\rest}{\!\!\restriction}

\renewcommand{\Re}{\mathop{\rm Re}\nolimits}
\renewcommand{\Im}{\mathop{\rm Im}\nolimits}
\newcommand{\ad}{\operatorname{ad}}

\newcommand{\neigh}{\operatorname{neigh}}

\newcommand{\Op}{{\operatorname{Op}^{{w}}_h}}

\renewcommand{\Re}{\mathop{\rm Re}\nolimits}
\renewcommand{\Im}{\mathop{\rm Im}\nolimits}

\newcommand{\ra}{\rangle}
\newcommand{\la}{\langle}

\newcommand{\Id}{{\rm Id}}
\newcommand{\DOCh}{{\mathcal R}(\delta,M_0,h)}

\def\hto0{\xrightarrow{h\to 0}}

\usepackage{amsxtra}

\ifx\pdfoutput\undefined
  \DeclareGraphicsExtensions{.pstex, .eps}
\else
  \ifx\pdfoutput\relax
    \DeclareGraphicsExtensions{.pstex, .eps}
  \else
    \ifnum\pdfoutput>0
      \DeclareGraphicsExtensions{.pdf}
    \else
      \DeclareGraphicsExtensions{.pstex, .eps}
    \fi
  \fi
\fi

\title[From open quantum systems to open quantum maps]
{From open quantum systems to open quantum maps}
\author[S. Nonnenmacher]
{St\'ephane Nonnenmacher}
\author[J. Sj\"ostrand]
{Johannes Sj\"ostrand}
\author[M. Zworski]
{Maciej Zworski}
\address{Institut de Physique Th\'eorique\\
CEA/DSM/PhT, Unit\'e de recherche associ\'ee au CNRS\\
CEA-Saclay\\
91191 Gif-sur-Yvette, France}
\email{snonnenmacher@cea.fr}
\address{Institut de Math\'ematiques de Bourgogne, UFR Science
et Techniques, 9 Avenue Alain Savary -- B.P. 47870, 21078 Dijon 
CEDEX, France} 
\email{jo7567sj@u-bourgogne.fr}
\address{Mathematics Department, University of California \\
Evans Hall, Berkeley, CA 94720, USA}
\email{zworski@math.berkeley.edu}

\begin{document}    

\maketitle   


\section{Introduction and statement of the results}
\label{int}

In this paper we show that for a class of open quantum systems
satisfying a natural dynamical assumption (see \S \ref{da})
the study of the resolvent, and hence of scattering, and of 
resonances, can be reduced, in the semiclassical limit, to the study of open quantum maps,
that is of finite dimensional quantizations of canonical 
relations obtained by truncation of symplectomorphisms derived from
the classical Hamiltonian flow (Poincar\'e return maps).

\begin{figure}
\begin{center}
\includegraphics[width=.9\textwidth]{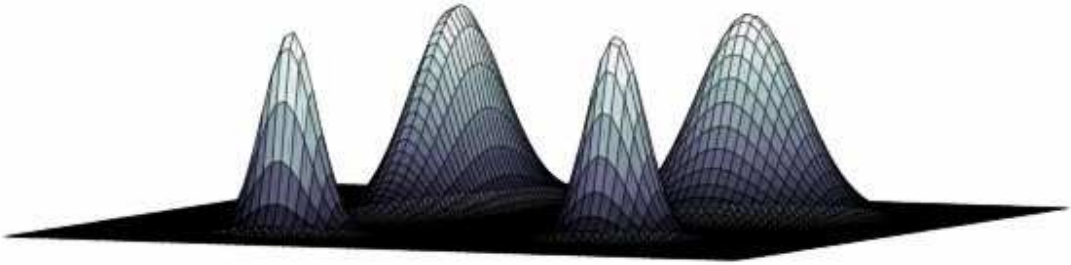}
\caption{\label{f:pot4} An example of a potential, $ V \in \CIc ( \RR^2 ) $,
to which the results apply: the Hamiltonian flow is hyperbolic on 
the trapped set in a range of energies -- see \cite[Appendix c]{SjDuke}.
In this example each energy surface $ p^{-1} ( E ) $ is
three dimensional, so the Poincar\'e section is two dimensional
as shown in Fig.~\ref{f:poi}.}
\end{center}
\end{figure}

We first explain the result in a simplified setting. For that 
consider the Schr\"odinger operator
\begin{equation}
\label{eq:P2}
P(h) =  - h^2 \Delta + V( x ) - 1  \,, \ \ V \in \CIc ( \RR^n ) \,, 
\end{equation}
and let $ \Phi^t $ be the corresponding classical flow on $T^*\IR^n\ni (x,\xi)$:
\begin{gather*}
\Phi^t ( x, \xi ) \defeq ( x( t) , \xi ( t ) ) \,, \\
x'( t ) =  2 \xi ( t) \,, \ \ \xi'( t ) = - dV ( x( t ) ) \,, \ \
x ( 0 ) = x \,, \ \ \xi( 0 ) = \xi \,.
\end{gather*}
Equivalently, this flow is generated by the Hamilton vector field
\be\label{e:Hp}
H_p(x,\xi)=\sum_{j=1}^n\frac{\partial p}{\partial \xi_j}
\frac{\partial}{\partial x_j} - \frac{\partial p}{\partial x_j}
\frac{\partial}{\partial \xi_j}
\ee
associated with the classical Hamiltonian
\be\label{e:symbol-p}
p(x,\xi)=|\xi|^2+V(x)-1\,.
\ee
The energy shift by $-1$ allows us to focus on the quantum
and classical dynamics near the energy $E = 0$, which will make our
notations easier\footnote{There is no loss of generality in this
  choice: the dynamics of the Hamiltonian $\xi^2+\tilde V(x)$ at some
  energy $E>0$ is equivalent with that of $\xi^2 + \tilde V/E -1$ at
  energy $0$, up to a time reparametrization by a factor
  $\sqrt{E}$. The same rescaling holds at the quantum level.}. 
We assume that the Hamiltonian flow has no fixed point at
this energy: $ dp \rest_{p^{-1} ( 0 ) } \neq 0 $.

The trapped set at any energy $ E $ is defined as 
\begin{equation}
\label{eq:K0} K_E\defeq \{ ( x , \xi )\in T^*\RR^n \; : \;  
p ( x, \xi ) = E \,, \  \Phi^t( x, \xi) \text{ remains 
bounded for all }t\in\IR \} \,. 
\end{equation}
The information about spectral and scattering properties 
of $ P=P(h) $ in \eqref{eq:P2} can be
obtained by analyzing the resolvent of $ P $, 
\[ R ( z ) = ( P - z )^{-1} \,, \ \ \ \Im z > 0 \,, \]
and its meromorphic continuation -- see for instance
\cite{PZ} and references given there. More recently semiclassical 
properties of the resolvent have been used to obtain local smoothing
and Strichartz estimates, leading to applications to nonlinear 
evolution equations -- see \cite{BGH} for a recent result and for 
pointers
to the literature. In the physics literature the Schwartz 
kernel of $ R ( z )$ is referred to as Green's function of the
potential $ V $.

The operator $P$ has absolutely continuous spectrum on the interval
$[-1,\infty)$; nevertheless, its resolvent $ R ( z ) $ continues 
meromorphically from $ \Im z > 0 $ to the disk $D ( 0 , 1 )$, in the 
sense that $ \chi R ( z ) \chi $, $ \chi \in \CIc( \RR^n ) $, 
is a meromorphic family of operators, with poles independent of
the choice of $ \chi \not \equiv 0 $ (see for instance \cite[Section 3]{SjZw91}
and \cite[Section 5]{Sj}).

The multiplicity of the 
pole $ z \in D ( 0 , 1 ) $  is given by 
\[  m_R ( z ) \defeq  {\rm{rank} }\, \oint_z \chi R ( w ) \chi dw \,, \]
where the integral runs over a sufficiently small circle around $ z $.

\begin{figure}
\begin{center}
\includegraphics[width=.7\textwidth]{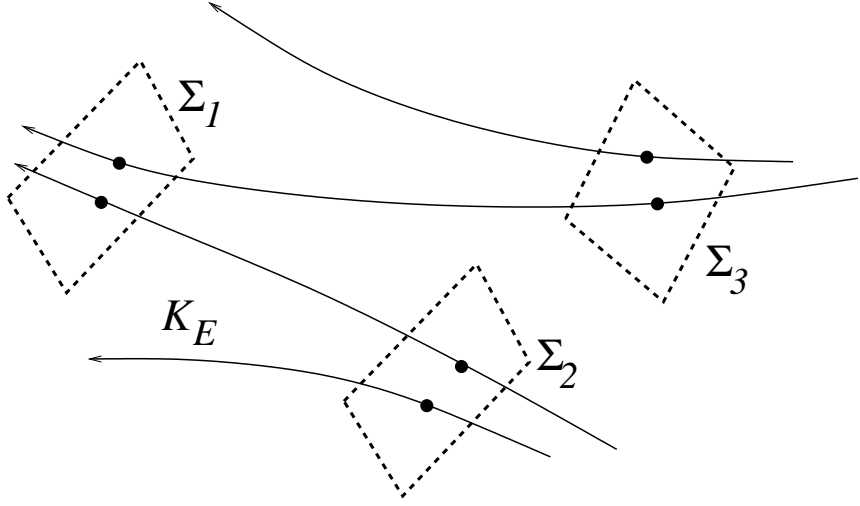}
\caption{\label{f:poi} A schematic view of
a Poincar\'e section $\bSigma=\sqcup_j \Sigma_j$ for $ K_E $ inside $ p^{-1} ( E ) $.
The flow near $ K_E $ can be described by an ensemble of symplectomorphisms
between different components $ \Sigma_j $ -- see \S \ref{da} for abstract
assumptions and a discussion why they are satisfied when the flow
is hyperbolic on $ K_E $ and $ K_E $ has topological dimension one.
The latter condition simply means that the intersections of $ K_E $ with
$ \Sigma_j$'s are totally disconnected.}
\end{center}
\end{figure}

We now assume that at energy $E=0$, the 
flow $ \Phi^t $ is {\em hyperbolic} on the trapped set $ K_0 $ 
and that this set is {\em topologically one dimensional}. 
Hyperbolicity means \cite[Def. 17.4.1]{KaHa} that at any point $\rho=(x,\xi)\in K_0$ the 
tangent space to the energy surface 
splits into the  neutral ($\RR H_p ( \rho) $), stable ($E_\rho^- $), and unstable
($E_\rho^+$) directions:
\be\label{e:hyperb1} T_\rho p^{-1} ( 0 ) = \RR H_p ( \rho) \oplus
E_\rho^- \oplus E_\rho^+ \,,
\ee
this decomposition is preserved through the flow, and is
characterized by the following properties:
\be\label{e:hyperb2}
\exists\, C>0,\ \exists\, \lambda>0,\quad \| {d \exp tH_p(\rho)v} \| 
\leq C\,e^{-\lambda |t|}\| {v} \|
\,,\quad
\forall\, v\in E^{\mp}_\rho,\ \pm t>0\,.
\ee
When $ K_0 $ is topologically one dimensional we can find a
Poincar\'e section which reduces the flow near $ K_0 $ to 
a combination of symplectic transformations, called the Poincar\'e map
$F$: 
see Fig.\ref{f:poi} for a schematic illustration and \S\ref{da} for 
a precise mathematical formulation. The {\it structural stability} of
hyperbolic flows \cite[Thm. 18.2.3]{KaHa} implies that the above properties will also hold for
any energy $E$ in a sufficientlys short interval $[-\delta,\delta]$
around $E=0$, in particular the flow near $K_E$ can be described
through a Poincar\'e map $F_E$.

Under these assumptions, we are interested in semiclassically locating
the resonances of the operator $ P(h) $ in a  neighbourhood of this energy 
interval:
\[ \DOCh \defeq [-\delta, \delta]+ i [- M_0  h \log (1/h) , M_0  h \log ( 1/h ) ] \,, \]
where $\delta, M_0  $ are independent of $h\in (0,1]$. 
Here the 
$ h \log(1/h ) $-size neighbourhood is natural in view of results on 
resonance free regions in case of no trapping -- see \cite{Ma}.

To characterize the resonances in $ \DOCh $ we 
introduce a family of ``quantum propagators'' quantizing the
Poincar\'e maps $F_E$.
\begin{thm}\label{t:s}
Suppose that $ \Phi^t $ is hyperbolic on $ K_0 $ 
and that $ K_0 $ is topologically one dimensional. More generally,
suppose that $P(h)$ and $\Phi^t$ satisfy the assumptions of
\S\ref{ass}-\S\ref{da}.
 
Then, for any $\delta>0$ small enough and any $M_0>0$, there exists
$h_0>0$ such that there exists a family of matrices, 
$$ \{ M ( z , h ),\ z \in \DOCh,\ h\in (0,h_0] \}\,, $$
holomorphic in the variable
$ z $,  and satisfying
\[   h^{-n+1} /C_0 \leq   \rank M ( z , h ) \leq  C_0 h^{-n+1 } \,,  \ \ C_0 > 1 \,, \]
such that for any $h\in (0,h_0]$, the zeros of 
\[   
\zeta ( z , h ) \defeq \det ( I - M ( z , h ) )  \,, 
\]
give the resonances of $ P(h) $ in $ \DOCh $, with correct multiplicities.

The matrices $ M ( z , h ) $ are {\it open quantum maps} associated
with the Poincar\'e 
maps $F_{\Re z}$ described above: for any $L>0$,
there exist a family of $h$-Fourier 
integral operators, $ \{\cM ( z , h ) \}$, quantizing the Poincar\'e 
maps $F_{\Re z}$ (see \S\ref{s:quantiz} and \S\ref{fio}), 
and projections $\Pi_h $ (see \S\ref{s:finite-dim})
of ranks
\[  h^{-n+1} /C_0
\leq \rank \Pi_h \leq  C_0 h^{-n+1 }  \,,  \]
such that 
\begin{equation}
\label{eq:Mzh}  M ( z , h ) =  \Pi_h \cM ( z , h ) \Pi_h  + \cO (h^L ) \,. 
\end{equation}
\end{thm}
The statement about the multiplicities in the theorem says that 
\be
\label{eq:mumu} \begin{split}
m_R ( z ) &  = \frac 1 { 2 \pi i } \oint_z \frac { \zeta ' ( w ) } {\zeta ( w ) } 
dw \\
& =  -\frac 1 { 2 \pi i }\tr  \oint_z ( I - M ( w ) )^{-1} M'( w ) dw \,. 
\end{split} 
\ee
A more precise version of Theorem \ref{t:s}, involving complex scaling and 
microlocally deformed spaces (see \S \ref{cs} and \S \ref{cef} 
respectively), will be given in Theorem \ref{t:mg} in \S \ref{igp}.
In particular Theorem~\ref{t:mg} gives us a full control over both the 
cutoff resolvent of $ P $,  $ \chi R ( z ) \chi $, and the full
resolvent $ ( P_{\theta,R} - z )^{-1} $ of the complex scaled operator $ P_{\theta,R} $,
in terms of the family of matrices $M(z,h)$; for this
reason, the latter is often called an {\em effective Hamiltonian} for $P$.

The mathematical applications of Theorem \ref{t:s} and its
refined version below include
simpler proofs of fractal Weyl laws \cite{SjZw04} and of the 
existence of resonance free strips \cite{NZ2}. The advantage lies
in eliminating flows and reducing the dynamical analysis to that
of maps. That provides an implicit second microlocalization without
any technical complication (see \cite[\S 5]{SjZw04}). The key is
a detailed understanding of the operators $ \cM (z,h ) $ stated in
the theorem.

\medskip

\noindent
{\em Relation to semiclassical trace formul{\ae}.}
The notation $\zeta(z,h)$ in the above theorem hints at the resemblance
between this determinant and a {\it semiclassical zeta function}. 
Various such functions have been introduced 
in the physics literature, to provide 
approximate ways of computing eigenvalues 
and resonances of quantum chaotic systems 
-- see \cite{Voros88,Gutz90,CvRosVatRug93}.

These semiclassical zeta functions are defined through formal
manipulations 
starting from the Gutzwiller trace formula -- see 
\cite{SjZw02} for a mathematical treatment and references.
They are given by 
sums, or Euler products, over periodic orbits where each term, or factor is an 
asymptotic series in powers of $h$. Most studies
have concentrated on the zeta function defined by the principal term, 
without $h$-corrections, which strongly resembles
the Selberg
zeta function defined for surfaces of constant negative curvature. 
However, unlike the case of the Selberg zeta function,
there is no known rigorous connection between the zeroes of the
semiclassical zeta function and the exact eigenvalues or resonances 
of the quantum system, even in the semiclassical limit. 
Nevertheless, 
numerical studies have indicated that the semiclassical zeta function admits
a strip of holomorphy beyond the axis of absolute convergence, and that its zeroes there are close to actual 
resonances \cite{CvRosVatRug93,wirzba}.

The traces of  $M(z,h)^k$, $k\in\NN$ admit semiclassical expressions
as sums over periodic points, which leads to a {\em formal}
representation of
$$
\zeta(z,h)=\exp\Big\{-\sum_{k=1}^\infty \frac{\tr M(z,h)^k}{k}\Big\}
$$ 
as a product over periodic points. That gives it the 
same form as the semiclassical zeta functions in the physics literature. 
In this sense, the 
function $\zeta(z,h)$ is a resummation of these formal expressions. As
will become clear from its construction below, the operator
$M(z,h)$ is not unique: it depends on many choices which 
affect the remainder term $\cO(h^L)$ in \eqref{eq:Mzh}.
However, the 
zeroes of $ \zeta ( z , h ) $ in $\DOCh$ 
are the exact resonances of the quantum Hamiltonian.


\medskip

\noindent
{\em Comments on quantum maps in the physics literature.}
Similar methods of analysis have been introduced in the theoretical physics
literature devoted to quantum chaos. The classical case involves a 
reduction to the boundary for obstacle problems: when the obstacle
consists of several strictly convex bodies, none of which intersects
a convex hull of any other two bodies, the flow on the trapped set is hyperbolic. The 
reduction can then be made to boundaries of the convex bodies, 
resulting with operators quantization Poincar\'e maps -- see 
Gaspard and Rice \cite{GaRi},
and for a 
mathematical treatment G\'erard \cite{Ge}, in the case of two convex bodies, 
and \cite[\S 5.1]{NSZ2}, for the
general case. Fig.\ref{f:ott} illustrates the trapped 
set in the case of three discs. The semiclassical analogue 
of the two convex obstacle, a system with one closed hyperbolic
orbit, was treated by G\'erard and the second author in \cite{GeSj}.
The approach of that paper 
was also based on the quantization of the Poincar\'e map near this orbit.

\begin{figure}
\begin{center}
\includegraphics[width=.9\textwidth]{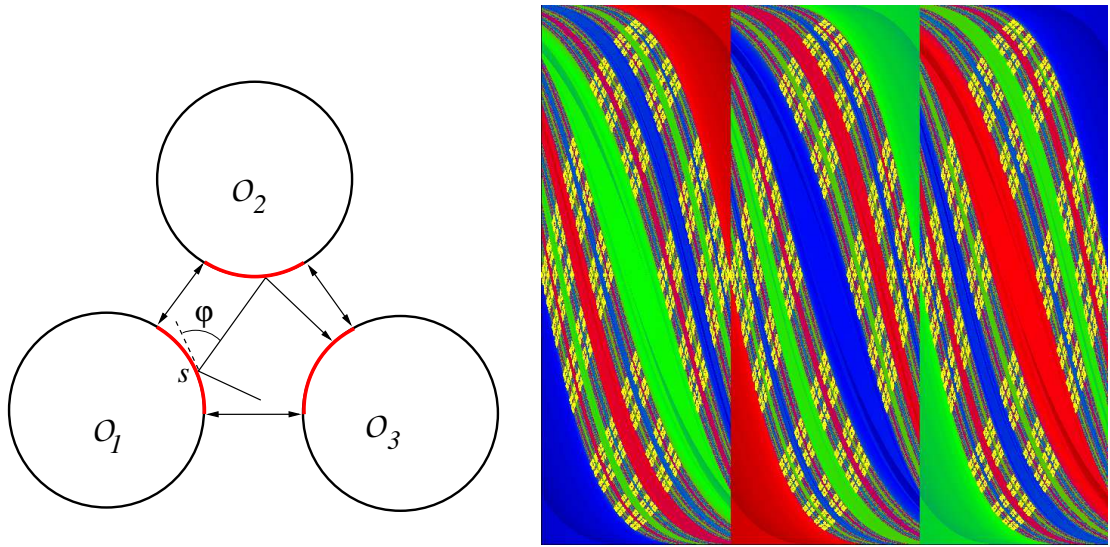}
\caption{\label{f:ott} This figure, taken from \cite{Ott}, 
shows the case of symmetric three disc scattering
problem (left), and the associated
Poincar\'e section (right). The section is the union of the three coball bundles
of circle arcs (in red) parametrized by $ s $ (the length parameter on the circle,
horizontal axis), and $\cos \varphi $ (vertical axis), where $ \varphi $ is 
the angle between the velocity after impact and the tangent to the circle. Green, blue,red
strips correspond to different regions of forward escape; 
they are bounded by components of the stable manifold. 
The trapped
set, $ {\mathcal T} $, shown in yellow,
is the intersection of the latter with the unstable manifold.}
\end{center}
\end{figure}

A reduction of
a more complicated quantum system to 
a quantized Poincar\'e map was proposed in the physics literature.
Bogomolny \cite{Bogo} studied a 
Schr\"odinger operator $P(h)$ with discrete spectrum, 
and constructed a family of energy dependent quantum transfer operators $T(E,h)$, which are integral 
operators acting on a hypersurface
in the configuration space. These transfer operators 
are  asymptotically unitary as $h\to 0$. 
The eigenvalues of $P(h)$ are then obtained, in the semiclassical limit,
as the roots of the equation $\det(1-T(E))=0$. Smilansky and co-workers 
derived a similar 
equation in the case of closed Euclidean 2-dimensional billiards 
\cite{DorSmil92}, replacing $T(E)$ by a (unitary) scattering matrix $S(E)$
associated with the dual scattering problem. Prosen \cite{Prosen95}
generalized Bogomolny's approach to a nonsemiclassical setting.
Bogomolny's
method was also extended to study quantum scattering situations \cite{GeorPran95,OzoVall99}.

Open quantum maps have first been defined in the quantum chaos literature
as toy models for open quantized chaotic systems
(see \cite[\S 2.2]{NZ1}, \cite[\S 4.3]{NZ12} and references given there). 
They generalized
the unitary quantum maps used to mimic bound chaotic systems \cite{DEGr03}.
Some examples of open quantum maps on the 2-dimensional torus or the 
cylinder,
have been used as models in various physical settings: 
Chirikov's quantum standard map (or quantum kicked rotator) was first 
defined in the context of plasma physics,
but then used as well to study ionization of atoms or molecules 
\cite{chirikov81}, as well as transport properties 
in mesoscopic quantum dots \cite{tworzydlo03}. Other maps, 
like the open baker's map, were
introduced as clean model systems, for which the classical dynamics is well
understood \cite{SaVa96,NZ12}. 
The popularity of quantum maps mostly stems from the much simplified 
numerical study they offer, both at the quantum and classical levels, 
compared with the case
of Hamiltonian flows or the corresponding Schr\"odinger operators.  For instance, the distribution
of resonances and resonant modes 
has proven to be much easier to study numerically for open quantum maps, than for realistic flows 
\cite{BorGuarShep91,schomerus,NZ1,KNPS06,NoRu07}.
Precise mathematical definitions of quantum maps on the torus phase space
are given in \cite[\S 4.3-4.5]{NZ1}.

\medskip

\noindent
{\em Organization of the paper.} In the remainder of this section 
we give assumptions on the operator $ P $ and
on the corresponding classical dynamical system, in particular we introduce
a Poincar\'e section $\bSigma$ and map associated with the classical
flow. We refer to 
results of Bowen and Walters \cite{BowenWalters72} 
to show that these assumptions are satisfied if the trapped set
supports a hyperbolic flow, and is topologically one dimensional, which 
is the case considered in Theorem~\ref{t:s}.

In \S \ref{pr} we recall various tools needed in our proof:
pseudodifferential calculus, the concept of semiclassical 
microlocalization, local $h$-Fourier integral operators associated
to canonical tranformation (these appear
in Theorem \ref{t:s}), complex scaling (used to define resonances as eigenvalues
of nonselfadjoint Fredholm operators), microlocally deformed
spaces, and Grushin problems used to define the effective Hamiltonians.

In \S \ref{gps} we follow a modified strategy of \cite{SjZw02} and construct
a microlocal Grushin problem associated with the Poincar\'e map on $\bSigma$.
No knowledge of that paper is a prerequisite but the self-contained discussion 
of the problem for the explicit case of $ S^1 $ given in \cite[\S 2]{SjZw02} can illuminate the complicated procedure presented here. In 
\cite[\S 2]{SjZw02} one finds the proof of the classical Poisson formula
using a Grushin problem approach used here. 

Because of the hyperbolic nature of the flow the microlocal Grushin problem
cannot directly be made into a globally well-posed problem -- see the remark
at the end of \S \ref{gps}. This serious difficulty is overcome in 
\S \ref{wpg} by 
adding microlocal weights adapted to the flow. This and suitably 
chosen finite dimensional projections lead to a well posed Grushin 
problem, with an effective Hamiltonian essentially given by a 
quantization of the Poincar\'e map: this fact is summarized in
Theorem~\ref{t:mg}, 
from which Theorem~\ref{t:s} is a simple corollary.

\medskip

\noindent
{\sc Acknowledgments.}
We would like to thank the National Science Foundation
for partial support under the grant 
DMS-0654436. 
This article was completed while the first author was visiting the
Institute of Advanced Study in Princeton, supported by the National
Science Foundation under agreement No. DMS-0635607.  The first and
second authors were also partially supported by the Agence Nationale
de la Recherche under the grant ANR -09-JCJC-0099-01.
Thanks also to Edward Ott for his permission to include Fig.\ref{f:ott} in 
our paper.

\section{Assumptions on the operator and on classical dynamics}
\label{assu}

Here we carefully state the needed assumptions on quantum
and classical levels.

\subsection{Assumptions on the quantum Hamiltonian $P(h)$}\label{ass}
Our results apply to operators $P(h)$ satisfying general assumptions given in 
\cite[\S 3.2]{NZ2} and \cite[(1.5),(1.6)]{SjZw04}. In particular,
they apply to certain elliptic differential operators on manifolds $ X $ of the form
\[  X = X_R \sqcup \bigsqcup_{j=1}^J \left( 
\RR^n \setminus \overline {B_{\RR^n} ( 0 , R )} \right) \,, \]
where $ R > 0 $ is large and $ X_R $ is a compact subset of $ X $. 
The reader interested
in this higher generality should consult those papers.

Here we will recall these assumptions only in 
the (physical) case of differential operators on $ X=\RR^n $.
We assume that
\begin{equation}
\label{eq:gac}
 P(h) = \sum_{|\alpha | \leq 2 } a_\alpha ( x , h ) ( h D_x )^\alpha \,,
\end{equation}
where $a_\alpha ( x, h )$ are bounded in $\CI(\RR^n)$, 
$a_\alpha ( x, h ) = a_\alpha^0 ( x ) + \cO( h )$ in $ \CI $, and
$a_\alpha(x , h)=a_\alpha(x)$ is independent of $h$ for $|\alpha|=2$.
Furthermore, for some $C_0 > 0 $ the functions 
$ a_\alpha ( x , h ) $ have holomorphic extensions to
\begin{equation}
\label{ge.2} \{x\in {\CC}^n \; : \; |\Re x|>C_0 \,, \ \  |\Im x|<|\Re x|/C_0\} \,,  
\end{equation} 
they are bounded uniformly with respect to $h$, and 
$a_\alpha ( x, h ) = a_\alpha^0 ( x ) + \cO( h )$
on that set.

Let $P(x,\xi )$
denote the (full) Weyl symbol of the operator $P$, so that
$P=P^w(x;hD;h)$, and assume 
\begin{equation}
\label{ge.3}
{P(x,\xi ;h)\to \xi ^2-1}
\end{equation}
when $x\to \infty $ in the set (\ref{ge.2}), uniformly with respect to
$(\xi ,h)\in K\times ]0,1]$ for any compact set $K\Subset \RR^n$
(here, and below, $ \Subset $ means that the set on the left is 
a pre-compact subset of the set on the right). 
We also assume that $P$ is classically elliptic:
\begin{equation}
\label{ge.4}
{
p_2 (x,\xi ) \defeq \sum_{|\alpha | =2}a_\alpha (x)\xi ^\alpha \ne 0\hbox{ on }
T^* \RR^n\setminus \{0\},
}
\end{equation}
and that $P$ is self-adjoint on $L^2(\RR^n)$ with domain
$H^2(\RR^n)$. The Schr\"odinger operator \eqref{eq:P2} corresponds to
the choices $\sum_{|\alpha|=2}a_{\alpha}\xi^\alpha=|\xi|^2$,
$a_\alpha\equiv 0$ for $|\alpha|=1$, and $a_0(x)=V(x)-1$. The
assumption \eqref{ge.3} show that we can also consider a slowly decaying
potential, as long as it admits a holomorphic extension in \eqref{ge.2}.

\subsection{Dynamical Assumptions}
\label{da}
The dynamical assumptions we need roughly mean that the flow $\Phi^t$ on the energy shell $p^{-1}(0)\subset T^*X$ 
can be encoded
by a Poincar\'e section, the boundary of which does not intersect the trapped set $K_0$. 
The abstract assumptions below are satisfied 
when the flow is hyperbolic on the trapped set which is assumed to be
topologically one dimensional -- see Proposition \ref{p:bowen}.

To state the assumption precisely, we notice that 
\ekv{ge.6}{p(x,\xi )=\sum_{|\alpha|\leq 2} a_\alpha ^0(x)\xi ^\alpha} 
is the semi-classical
principal symbol of the operator $P(x,hD;h)$.
We assume that the characteristic
set of $ p $ (that is, the energy surface $p^{-1}(0)$) is a simple hypersurface:
\begin{equation}\label{ge.8}
dp\ne 0 \hbox{ on }p^{-1}(0).
\end{equation}
Like in the introduction, we denote by 
$$
\Phi^t\defeq \exp(tH_p):T^*X\to T^*X
$$ 
the flow generated by the Hamilton vector field $H_p$ (see \eqref{e:Hp}).
\renewcommand\thefootnote{\dag}%

Our assumptions on $p(x,\xi)$ ensure that, for $E$ close to $ 0 $,
we still have no fixed point in $p^{-1}(E)$, 
and the trapped set $K_E$ (defined in \eqref{eq:K0}) is a compact subset of $p^{-1}(E)$.

\medskip

We now assume that there exists a ``nice''
Poincar\'e section for the flow near $K_0$, namely finitely many
compact contractible smooth hypersurfaces $\Sigma _k\subset p^{-1}(0)$,
$k=1,2,\ldots,N$ with smooth boundaries, 
such that 
\begin{equation}
\label{eq:H1} \partial \Sigma _k\cap  K_0=\emptyset \,, \quad \Sigma_k\cap\Sigma_{k'}=\emptyset,\ \ k\neq k',
\end{equation}
\begin{equation}
\label{eq:H2} \text{ $H_p$ is transversal to $\Sigma _k$
uniformly up to the boundary,} 
  \end{equation}
\begin{gather} \label{eq:H3} \begin{gathered}
\text{For every $\rho \in K_0$, there
exist $ \rho _-\in \Sigma _{j_-(\rho )} \,, \quad \rho _+\in \Sigma _{j_+(\rho
)}$ } \\
\text{ of the form $\rho _\pm = \Phi^{\pm t_\pm ( \rho )}(\rho ) $, 
with $0<t_\pm (\rho ) \leq t_{\max} <\infty $, such that } \\
 \{ \Phi^{t}(\rho );\,
-t_-(\rho)<t<t_+(\rho),\ t\ne 0\} \cap \Sigma_k  = \emptyset 
\,, \ \ \forall \, k \,.  
\end{gathered}
\end{gather}
We call Poincar\'e section the disjoint union
$$
\bSigma\defeq \sqcup_{k=1}^N\Sigma _k\,.
$$
The functions $\rho\mapsto \rho _\pm(\rho )$, $\rho\mapsto t_\pm(\rho) $
are uniquely defined ($\rho_\pm(\rho)$ will be called respectively the successor and predecessor of $ \rho $).
They remain well-defined for $\rho$ in some neighbourhood of $K_0$ in $p^{-1}(0))$
and, in such a neighbourhood, depend smoothly on $\rho $ away from $\bSigma$. 
In order to simplify the presentation we also assume the successor
of a point $\rho\in\Sigma_k$ belongs to a different component:
\begin{equation}
\label{eq:H4} 
\text{ If $\rho \in \Sigma _k\cap K_0$ for some $k$,
then $\rho _+(\rho )\in \Sigma _\ell\cap K_0 $ for some $\ell
\ne k$. }
\end{equation}
The section can always be enlarged
to guarantee that this condition is satisfied. For instance, for $ K_0 $ consisting
of one closed orbit we only need one transversal component to have 
\eqref{eq:H1}-\eqref{eq:H2}; to fulfill \eqref{eq:H4} a second component has to be added. 

We recall that hypersurfaces in
$p^{-1}(0)$ that are transversal to $H_p$ are symplectic. In fact, 
a local application of Darboux's theorem (see for instance \cite[\S 21.1]{Hor2})
shows that we can make a symplectic change of variables in which $ p = \xi_n $
and $ H_p = \partial_{x_n} $. If $ \Sigma \subset \{ \xi_n = 0 \} $ is 
transversal to $ \partial_{x_n} $, then $ ( x_1, \cdots x_{n-1}; \xi_1 , \cdots,\xi_{n-1} ) $
can be chosen as coordinates on $ \Sigma $. Since $ \omega \rest _{p^{-1}(0)} 
= \sum_{j=1}^{n-1}d\xi_j \wedge dx_j$, that means that $ \omega \rest_\Sigma $ is
nondegenerate. The local normal form $ p = \xi_n $ will be used further in 
the paper (in its quantized form).

The final assumption guarantees the absence of topological or symplectic peculiarities:
\begin{equation}
\label{eq:H5}
\begin{split}
& \text{There exists a set $\wSi
_k\Subset T^* {\RR}^{n-1}$ with smooth boundary, and a symplectic}
\\
& \text{diffeomorphism  $\kappa _k:
\tSigma_k\to \Sigma_k$ which is smooth up the boundary together}
\\ & \text{with its inverse. We assume that $ \kappa_k $ extends to a 
neighbourhood of $\wSi_k$ in $T^*\RR^n$.}
\end{split}
\end{equation}
In other words, there exist symplectic coordinate charts on $\Sigma_k$, taking values in $\widetilde\Sigma_k$.

\medskip

The following result, due to Bowen and Walters \cite{BowenWalters72},
shows that our assumptions are realized in the case of 1-dimensional hyperbolic
trapped sets.
\begin{prop}
\label{p:bowen}
Suppose that the assumptions of \S\ref{ass} hold, 
and that the flow $ \Phi^t\rest_{K_0} $ is {\it hyperbolic} in the 
standard sense of Eqs.~(\ref{e:hyperb1},\ref{e:hyperb2}). 
Then the existence of $ \bSigma$
satisfying \eqref{eq:H1}-\eqref{eq:H5} is equivalent with $K_0$ being
topologically one dimensional.
\end{prop}

\smallskip
\noindent
{\bf Remark.} Bowen and Walters \cite{BowenWalters72} show more, 
namely the fact that the sets $\{\Sigma_k\}$ can be chosen of
small diameter, and constructed such that $\bSigma\cap K_0$
forms a {\em Markov partition} for the Poincar\'e map.
Small diameters ensures that \eqref{eq:H5} holds, while, as
mentioned before, \eqref{eq:H4} can always be realized by adding some more
components.

\medskip

Proposition~\ref{p:bowen} 
shows that the assumptions of Theorem \ref{t:s}
imply the dynamical assumptions made in this section. 
The proof of \cite[Appendix c]{SjDuke} shows that the following
example of ``three-bumps potential'',
\begin{gather*}
  P = - h^2 \Delta + V ( x ) - 1 \,, \ \ x \in \RR^2 \,, \ \ \ 
V ( x ) = 2\sum_{k=1}^3 \exp ( - R( x - x_k )^2 ) \,, \\ x_k = 
( \cos ( 2\pi k/3 ) , \sin ( 2 \pi k / 3 ) ) \,, 
\end{gather*}
satisfies our assumptions as long as $R>1$ is large enough (see Fig.~\ref{f:pot4}).

\begin{figure}
\begin{center}
\includegraphics[angle=-00,width=.95\textwidth]{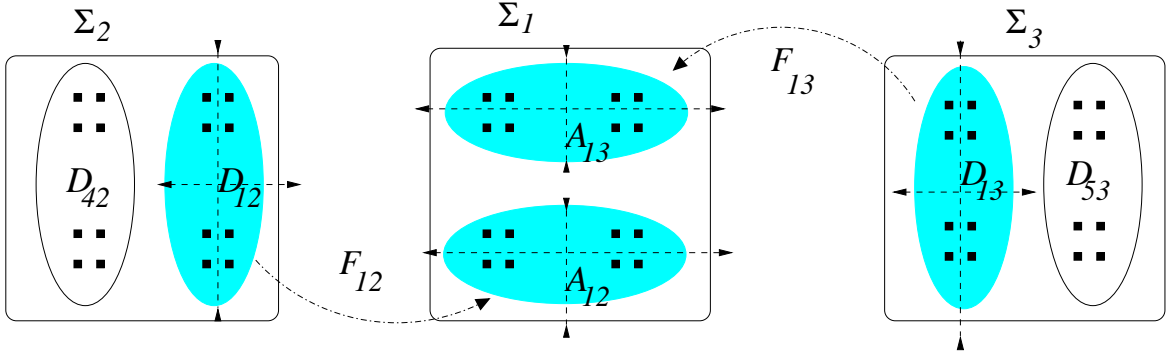}
\caption{\label{fig:A-D} Schematic representation of the components
  $F_{ik}$ of the Poincar\'e map between the 
sets $D_{ik}$ and $A_{ik}$ (horizontal/vertical ellipses). The reduced trapped
set $\TT_i$ is represented by the black squares. The unstable/stable directions of the map
are the horizontal/vertical dashed lines.}
\end{center}
\end{figure}

\subsection{The Poincar\'e map}
Here we will analyze the Poincar\'e map associated with the
Poincar\'e section discussed in 
\S\ref{da}, and its semiclassical quantization. 

\subsubsection{Classical analysis}\label{s:decompo}

The assumptions in \S\ref{da} imply the existence of a {\em symplectic
  relation}, the so-called Poincar\'e map on $\bSigma$.

More precisely, let us identify $ \Sigma_k $'s with $ \wSi_k $
using $ \kappa_k $ given in \eqref{eq:H5}, so that the Poincar\'e section 
\[ \bSigma =
\bigsqcup_{k=1}^N \Sigma_k \simeq \bigsqcup_{k=1}^N \widetilde 
\Sigma_k   \subset \bigsqcup_{k=1}^N T^* \RR^{n-1}\,. \]
Let us call
\[  \TT \defeq K_0 \cap \bSigma=\bigsqcup_k\TT_k \quad\text{the reduced trapped set.}  \]
The map
\[  f \; : \; \TT \longrightarrow \TT \,, \ \ \rho \longmapsto
f(\rho)\defeq\rho_+ ( \rho )  \]
(see the notation of \eqref{eq:H3}) is the Poincar\'e map for
$\Phi^t\rest_{K_0}$. It is a Lipschitz bijection. The decomposition 
$\TT=\bigsqcup_k\TT_k$ allows us to define the {\em arrival} and 
{\em departure} subsets of $\TT$:
\[ \begin{split}
& \cD_{ ik} \defeq \{ \rho \in \TT_k\subset\Sigma_k \; : \; \rho_+( \rho ) \in \TT_i \} = 
\TT_k \cap f ^{-1} (\TT_i )\,,\\
&  \cA_{ik } \defeq \{ \rho \in \TT_i\subset\Sigma_i \; : \; \rho_-( \rho ) \in \TT_k \} 
= \TT_i \cap f ( \TT_k ) =f(\cD_{ik})\,, 
 \end{split} \]
For each $k$ we call $J_+(k)\subset\{1,\ldots,N\}$ the set of indices $i$ 
such that $\cD_{ik}$ is not empty (that is, 
for which $\TT_i$ is a successor of $\TT_k$). Conversely, the set $J_-(i)$ refers to the
predecessors of $\TT_i$.

Using this notation, the map $f$ obviously decomposes into a family of Lipschitz 
bijections $f_{ik}:\cD_{ik}\to \cA_{ik}$. 
Similarly to the maps $\rho_{\pm}$, each $f_{ik}$ can be extended to a neighbourhood of $\cD_{ik}$, to form 
a family of local smooth symplectomorphisms 
$$
F_{ik} \; : \; D_{ ik} \longrightarrow \; F_{ik} ( D_{ik } ) \defeq A_{ik}\,, 
$$
where $D_{ik}$ (resp. $A_{ik}$) is a neighbourhood of $\cD_{ik}$ in $\Sigma_k$ (resp. a neighbourhood
of $\cA_{ik}$ in $\Sigma_i$).
Since our assumption on $K_0$ is equivalent with the fact that the reduced trapped
set $\TT$ is {\em totally disconnected},  we may assume that the sets $\{D_{ik}\}_{i\in J_+(k)}$ 
(resp. the sets $\{A_{ik}\}_{k\in J_-(i)}$) are mutually disjoint. We will call 
$$
D_k \defeq \sqcup_{i\in J_+(k)} D_{ik},\qquad A_i \defeq\sqcup_{k\in J_-(k)}A_{ik}\,.
$$
Notice that, for any index $i$, the sets $D_{i}$, $A_{i}$ both contain
the set $\TT_i$, so they are not disjoint.

We will also define the {\em tubes} $T_{ik}\subset T^*X$ containing the trajectories between $D_{ik}$ and $A_{ik}$:
\begin{equation}\label{e:tubes}
T_{ik}\defeq \{\Phi^t(\rho),\,:\,\rho\in D_{ik},\ 0\leq t\leq t_+(\rho)\}\,.
\end{equation}
See Fig.~\ref{fig:A-D} for a sketch of these definitions, and Fig.~\ref{fig:tube} for an artistic view
of $T_{ik}$
\begin{figure}
\begin{center}
\includegraphics[width=0.75\textwidth]{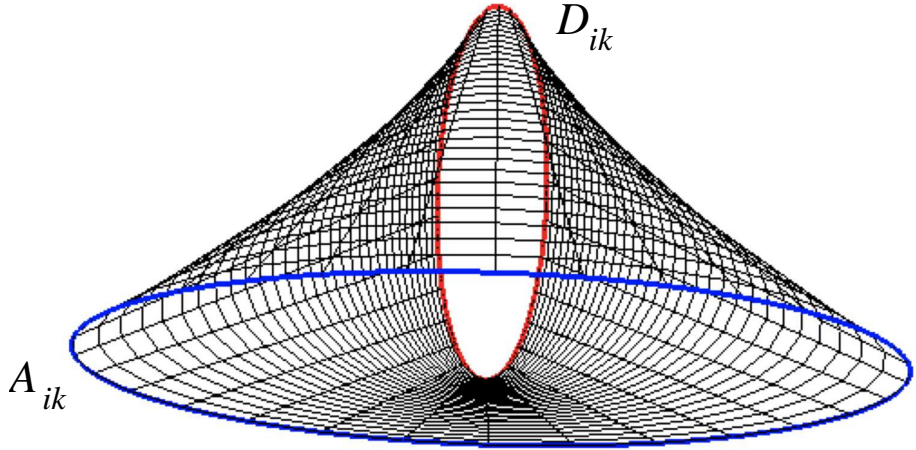}
\caption{\label{fig:tube} Trajectories linking the boundaries of the departure set 
$D_{ik}\subset\Sigma_k$ and the arrival set $A_{ik}\subset \Sigma_i$. Note the stretching and contraction implied by
hyperbolicity. These trajectories and $D_{ik}\cup A_{ik}$ form the boundary of the tube $T_{ik}$ defined by \eqref{e:tubes}.}
\end{center}
\end{figure}

The maps $F_{ik}$ will be grouped into the symplectic bijection $F$ 
between $\bigsqcup_k D_{k}$ and $\bigsqcup_k A_{k}$.
We will also call $F$ the Poincar\'e map, which can be viewed as a
symplectic relation on $\bSigma$. 
We will sometimes identify the map $F_{ik}$ with its action on 
subsets of $ T^* \RR^{n-1} $.
$$\tF_{ik}=\kappa_i^{-1}\circ F_{ik}\circ \kappa_{k}:\widetilde D_{ik} 
\longrightarrow \widetilde A_{ik}  \,, \ \ 
\widetilde D_{ik} \defeq \kappa_k^{-1} ( D_{ ik} ) \,, \ \ 
\widetilde A_{ik} \defeq \kappa_i^{-1} ( A_{ ik} )  \,.
$$

\medskip

Using the continuity of the flow $\Phi^t$, we will show in
\S\ref{s:normal} that the above structure
can be continuously extended to a small energy interval $z\in
[-\delta,\delta]$.
The Poincar\'e map
for the flow in $p^{-1}(z)$ will be denoted by $F_z=(F_{ik,z})_{1\leq i,k\leq N}$ (see \S\ref{s:normal} for details).

In the case of $K_0$ supporting a hyperbolic flow, a structural
stability of $K_z$ holds in a stronger sense: the flows $\Phi^t\rest_{K_z}$ and
$\Phi^t\rest_{K_0}$ are actually orbit-conjugate (that is, conjugate
up to time reparametrization) by a homeomorphism close to the identity.
\cite[Thm. 18.2.3]{KaHa}.

\subsubsection{Quantization of the Poincar\'e map}\label{s:quantiz}
In this section we make more explicit the operator $\cM(z,h)$ used in
Theorem~\ref{t:s}. The semiclassical tools we are using will be recalled in \S\ref{pr}.
 
Let us first focus on a single component $F_{ik}:D_{ik}\to A_{ik}$ 
of the Poincar\'e map. A quantization of the symplectomorphism $F_{ik}$ 
(more precisely, of its pullback $\tF_{ik}$) 
is a semiclassical (or $h$-) Fourier integral operator, that is a
family of operators
$\cM_{ik}(h):L^2(\IR^{n-1})\to L^2(\IR^{n-1})$, $h\in (0,1]$, 
whose semiclassical wavefront set satisfies
\begin{equation}
\label{eq:WFM}
  \WFh' ( \cM_{ik} ) \Subset \tA_{ik} \times \tD_{ik} \,,
\end{equation}
and which is associated with the symplectomorphism $\tF_{ik}$. 
($h$-FIOs are defined in \S\ref{fio}, and $ \WFh' $ is defined in \eqref{eq:wfhp} below).

Being associated to the symplectic map $ \tF_{ik} $ means the following thing:
for any $ a \in \CIc ( \tA_{ik} ) $, the quantum operator $\Op ( a )$
transforms as follows when conjugated by $\cM_{ik}(h)$
\begin{equation}
\label{eq:aM}
\cM_{ik}(h)^* \Op ( a ) \cM_{ik}(h) = \Op ( \alpha_{ik} \tF_{ik}^* a ) + h^{1-2\delta}\, \Op ( b) \,, \ 
\end{equation}
where the symbol $ \alpha_{ik} \in S_\delta(T^*\RR^{n-1}) $ is independent of $ a $, 
$ \alpha_{ik} = 1 $ on some neighbourhood of $\TT_k$ in  $\Sigma_k$, and $ b \in S_\delta(T^*\RR^{n-1}) $,
for every $ \delta > 0 $. Here $ \Op $ 
denotes the semiclassical Weyl quantization on $ \RR^{2(n-1)}$ (see eq.\eqref{eq:weyl}), and 
$ S_\delta(T^*\RR^{n-1})$ is the symbol class defined in \S \ref{sa}.
The necessity to have $ \delta>0 $ in \eqref{eq:aM} comes from the slightly 
exotic nature of our Fourier integral operator, due to the presence
of some mild exponential weights -- see \S\ref{cef} below. 

The property \eqref{eq:aM}, which is a form of {\em Egorov's theorem}, 
characterizes $ \cM_{ik}(h)$ as a semiclassical 
Fourier integral operator associated with $\tF_{ik}$
(see \cite[Lemma 2]{SjZw02} and \cite[\S 10.2]{EZ} for that characterization).

We can then group together the $\cM_{ik}(h)$ into a single operator-valued matrix 
(setting $\cM_{ik}(h)=0$ when $i\not\in J_+(k)$):
\[  \cM(h) \; : \; L^2 ( \RR^{n-1} )^N
\longrightarrow  L^2 ( \RR^{n-1} )^N \,, \quad \cM(h)=\begin{pmatrix}\cM_{ik}(h)\end{pmatrix}_{1\leq i,k\leq N}\,.
\]
We call this $\cM(h)$ a quantization of the Poincar\'e map $F$.

The operators $\cM(z,h)$ in
Theorem~\ref{t:s} will also holomorphically depend on $z\in \DOCh$,
such that for each $z\in \DOCh \cap \IR$ the family $(\cM(z,h))_{h\in (0,1]}$
is an $h$-Fourier integral operator of the above sense.


\medskip

\noindent
{\bf Comment on notation.}
Most of the estimates in this paper include error terms of 
the type $\cO(h^\infty )$, which is natural in all microlocal statements. 
To simplify the notation 
we adopt the following convention (except in places
where it could lead to confusion): 
\begin{gather}
\label{eq:not}
\begin{gathered}
u \equiv v  \ \Longleftrightarrow \  \| u - v \| =\Oo ( h^\infty) 
\| u \| \,, \\
\| S u \| \lesssim  \| T u \| + \|v \| \ \Longleftrightarrow 
\ \| S u \| \leq \Oo ( 1 ) 
 (  \| T u \| + \|v \| ) + \Oo ( h^\infty ) \| u \| \,,
\end{gathered}
\end{gather}
with norms appropriate to context.
Since most estimates involve
functions $ u $ microlocalized to compact sets, in the 
sense that, $ u - \chi ( x , h D ) u \in h^\infty {\mathscr S} ( \RR^n)  $, 
for some $ \chi \in \CIc ( T^* \RR^n) $, the norms are almost 
exclusively $ L^2 $ norms, possibly with microlocal weights described in 
\S \ref{cef}. 

The notation $ u = \cO_V ( f ) $
means that $ \| u \|_{V } = \cO( f ) $, and the notation
$ T = \cO_{V \rightarrow W } ( f ) $ means that
$ \| T u \|_W = \cO (f ) \| u \|_V $.
Also, the notation 
\[ \neigh(A,B)  \ \ \text{ for $ A \subset B $,} \]
means an open neighbourhood of the set $ A $ inside the set $ B $.

Starting with \S \ref{gps}, we denote the Weyl quantization
of a symbol $ a $ by the same letter $ a = a^w ( x, h D ) $.
This makes the notation less cumbersome and should be clear
from the context.

Finally, we warn the reader that 
from \S \ref{gps} onwards the original operator $ P $
is replaced by the complex scaled operator $ P_{\theta , R } $, whose 
construction is recalled in \S \ref{cs}. Because of the
formula \eqref{eq:two_r}, that does not affect the results formulated
in this section.

\section{Preliminaries}
\label{pr}
In this section we present background material and references needed for the
proofs of the theorems.

\subsection{Semiclassical pseudodifferential calculus}
\label{sa}

We start by defining a rather general class of 
symbols (that is, $h$-dependent functions) on the phase space
$T^*\RR^d$. For any $\delta\in [0,1/2]$ and $m,k\in\RR$, let
\[ \begin{split} 
S^{m,k}_\delta ( T^* \RR^d ) =\big\{ & a \in \CI( T^* \RR^d \times (0, 1]  ) :
\forall \, \alpha \in \NN^d \,, \  \beta \in \NN^d\,,    \ \exists\, 
C_{\alpha \beta}>0\,,  \\
& \ 
 |\partial_x ^{ \alpha } \partial _\xi^\beta a ( x, \xi ;h ) | \leq
C_{\alpha \beta}  h^{-k-\delta ( | \alpha| + |\beta |) }
\langle \xi \rangle^{m-|\beta| } \big\}   \,.
\end{split} \]
where
$ \langle \xi \rangle \defeq ( 1 + |\xi|^2 )^{\frac12} $.

Most of the time we will use the class with $ \delta = 0 $
in which case we drop the subscript. When $ m = k =0$,
we simply write $ S ( T^*\RR^d ) $ or $ S $ for the class of symbols. 
In the paper $ d = n $ (the dimension of the physical space) or $ d = n-1 $ 
(half the dimension of the Poincar\'e section),
and occasionally (as in \eqref{eq:WFM}) $ d= 2n-2 $, 
depending on the context.

The quantization map, in its different notational guises, is defined as follows 
\be\label{eq:weyl}
\begin{split}
a^w u & =  \Op (a) u ( x) = a^w(x,hD) u(x) 
\\
& \defeq \frac1{ ( 2 \pi h )^d } 
  \int \int  a \big( \frac{x + y  }{2}  , \xi \big) 
e^{ i \la x -  y, \xi \ra / h } u ( y ) dy d \xi \,, 
\end{split} 
\ee
and we refer to \cite[Chapter 7]{DiSj} for a detailed discussion of 
semiclassical quantization (see also \cite[Appendix]{SjR}), and 
to \cite[Appendix D.2]{EZ} for the semiclassical calculus for
the symbol classes given above.

We denote by
$ \Psi_{\delta}^{m,k} ( \RR^d) $  or $\Psi^{m,k} ( \RR^d )$
the corresponding classes of pseudodifferential operators. The
quantization formula \eqref{eq:weyl} is bijective:
each operator $A\in \Psi_{\delta}^{m,k} ( \RR^d)$ is exactly
represented by a unique (full) symbol $a(x,\xi;h)$. It is useful to 
consider only certain equivalence classes of this full symbol, thus defining a
{\em principal symbol map} -- see \cite[Chapter 8]{EZ}: 
$$
   \sigma_h \; : \; \Psi_\delta^{m,k} ( \RR^d  ) \ \longrightarrow 
  S_\delta^{m,k} ( T^* \RR^d ) / S_\delta^{m-1,k-1+2\delta} ( T^*
  \RR^d ) \,.
$$
The combination $\sigma_h \circ \Op $ is the
natural projection from $S_\delta^{m, k}$ onto $S_\delta^{m, k }  /
S_\delta^{m-1, k-1+2\delta} $. The main property of this principal symbol map is
to ``restore commutativity'':
$$
\sigma_h ( A \circ B ) = \sigma_h ( A )\sigma_h ( B ) \,.
$$
Certain symbols in $S^{m,0}(T^*\RR^d)$ admit an asymptotic expansion in powers of $h$,
\be\label{e:classical-type}
a(x,\xi;h)\sim \sum_{j\geq 0} h^j\,a_j(x,\xi),\qquad a_j\in S^{m-j,0}\
\text{independent of $h$}\,,
\ee
such symbols (or the corresponding
operator) are called {\em classical}, and make up the subclass
$S_{cl}^{m,0}(T^*\RR^d)$ (the corresponding operator class is denoted by $\Psi_{cl}^{m,0}(\RR^d)$).
For any operator $A\in \Psi_{cl}^{m,0}(\RR^d)$, its principal symbol $\sigma_h(A)$ admits as representative
the $h$-independent function $a_0(x,\xi)$, first term in
\eqref{e:classical-type}. The latter is also usually
called the principal symbol of $a$.

In \S\ref{cef} we will introduce a slightly different notion of {\em
  leading symbol}, adapted to a subclass of symbols in $S(T^*\RR)$
larger than $S_{cl}(T^*\RR^d)$.


The semiclassical Sobolev spaces, $ H_h^s ( \RR^d ) $  are defined 
using the semiclassical Fourier transform, $ {\mathcal F}_h $:
\begin{equation}
\label{eq:Sob}
  \| u \|^2_{ H_h^s } \defeq  \int_{\RR^d} 
\langle \xi \rangle^{2s} |{\mathcal F}_h u ( \xi ) |^2 d \xi \,, \ \ 
\ {\mathcal F}_h u ( \xi ) \stackrel{\rm{def}}{=} 
\frac{1}{ ( 2 \pi h )^{d/2} }\int_{\RR^d } u ( x ) e^{ - i \langle x , \xi \rangle/h } d x \,.
\end{equation}
Unless otherwise stated all norms in this paper, $ \| \bullet \| $, 
are $ L^2 $ norms.

We recall that 
the operators in $ \Psi ( \RR^d ) $ are bounded on $ L^2 $ uniformly 
in $ h $, 
and that they can be characterized using commutators by 
Beals's Lemma (see \cite[Chapter 8]{DiSj} and \cite[Lemma 3.5]{SjZw04} for the
$ S_\delta $ case):
\begin{equation}
\label{eq:beals} A \in \Psi_{\delta} ( X  )
\; \Longleftrightarrow \; \left\{ \begin{array}{l} \| \ad_{\ell_N}
\cdots \ad_{ \ell_1  } A \|_{ L^2 \rightarrow L^2 } = \cO ( h^{(1-\delta) N } )
 \\
\text{ for linear functions $ \ell_j ( x, \xi ) $ on $ \RR^d \times \RR^d$,}
\end{array} \right. \end{equation}
where $ \ad_B A = [ B , A ] $.

For a given symbol $ a \in S ( T^* \RR^d ) $ we follow \cite{SjZw02} and 
say that the {\em essential support} is contained in a given compact
set $ K \Subset T^*\RR^d $, 
$$
  {\text{ess-supp}}_h\; a \subset K \Subset T^*\RR^d\,, 
$$
if and only if   
$$
 \forall \, \chi \in S ( T^*\RR^d ) \,, \ \supp \chi \cap
 K  = \emptyset \ \Longrightarrow
 \ \chi \, a \in h^\infty \cS ( T^* \RR^d) \,.
$$
The essential support is then the intersection of all such $ K$'s.

Here $\cS $
denotes the Schwartz space.
For $ A \in \Psi ( \RR^d) $,  $  A = \Op ( a ) $, we call
\begin{equation}
\label{eq:WFA}
    \WFh ( A) =  \text{ess-supp}_h\; a \,.
\end{equation}
the semiclassical wavefront set of $A$. 
(In this paper we are concerned with a purely semiclassical 
theory and will only need to deal with {\em compact} subsets of $ T^* \RR^d $.
Hence, we won't need to define noncompact essential supports).


\subsection{Microlocalization}
\label{mic}

We will also consider spaces of $ L^2 $ functions (strictly speaking, 
of $ h $-dependent families of functions) which are 
{\em microlocally concentrated} in an open set 
$ V  \Subset T^* \RR^d$:
\begin{equation}
\label{eq:HV} 
\begin{split}
 H( V ) \defeq  
\{ & u = (u ( h )\in L^2(\RR^d))_{h\in(0,1]},  \ \ \text{such that} \\
&  \exists \, C_u  > 0\,, \  \forall \, h \in ( 0, 1 ] 
\,, \quad\| u ( h ) \|_{L^2 ( \RR^d) } \leq C_u \,,
\\
& \exists \, \chi \in \CIc ( V) \,, \quad
\chi^w(x,hD_x)\, u(h) =  u(h) + \cO_{\cS}  ( h^\infty )
\} \,. \end{split}
\end{equation}
The semiclassical wave front set of $ u \in H ( V ) $ is defined as:
\be\label{eq:defWF}
\WFh ( u ) =   T^* \RR^d \setminus  \big\{  ( x, \xi )\in T^*\RR^d
\; : \; \exists \, a \in S ( T^* \RR^d ) \,, \  \ a ( x, \xi ) =1
\,, \ \| a^w \, u \|_{L^2} = \cO( h^\infty) \big\}
 \,.
\ee
The condition \eqref{eq:defWF} can be equivalently replaced
with $ a^w\, u = \cO_{\cS} ( h^\infty)$,
since we may always take $ a \in \cS ( T^* \RR^d ) $.
This set obviously satisfies $ \WFh ( u ) \Subset V $. Notice that
the condition does not characterize the individual functions $u(h)$, but the full
sequence as $h\to 0$.

We will say that an $ h$-dependent family of operators $T=(T(h))_{h\in(0,1]}:\cS(\RR^d)\to \cS'(\RR^k)$ 
is {\em semiclassically tempered} if there exists $L\geq 0$ such that 
$$
\|\la x\ra^{-L}T(h) u\|_{H_h^{-L}}\leq C\,h^{-L}\| \la x
\ra^Lu\|_{H_h^{L}}\,,\quad h\in (0,1)\,,
\quad \la x\ra \defeq (1+x^2)^{1/2}\,.
$$
Such a family of operators is {\em microlocally defined} on $ V $ if one only
specifies (or considers) its action on states $u\in H(V)$, modulo $\cO_{\cS'\to \cS}(h^\infty)$.
For instance, $T$ is said to be asymptotically uniformly bounded on $H(V)$ if
\begin{equation}\label{e:asymp-bounded}
\exists \, C_T>0\ \forall \, u\in H(V) \ \exists \, h_{T,u}>0\,, \
\forall \,  h \in ( 0 , h_{T,u})  \,, \quad
\| T ( h ) u(h) \|_{L^2 ( \RR^k ) } \leq C_T\,C_u \,.
\end{equation}
Two tempered operators $T,T'$ are said to be microlocally equivalent on
$V$, iff for any $u\in H(V)$ they satisfy
$\| (T-T')u \|_{L^2(\RR^k)}=\Oo(h^\infty)$; equivalently, for any
$\chi\in \CIc(V)$, $\| (T-T')\chi^w \|_{L^2\to L^2}=\Oo(h^\infty)$.


If there exists an open subset $ W \Subset T^* \RR^k $ and $L\in\IR$ such 
that $T$ maps any $u\in H(V)$ into a state $Tu\in h^{-L}\,H(W)$, then we will 
write
\[ 
T = T ( h ) \; : \; H ( V ) \longrightarrow H( W ) \,,
\]
and we say that $T$ is defined microlocally
in $W\times V$.


For such operators, we may define only the part of the (twisted) wavefront set which is inside $W\times V$:
\begin{equation}
\label{eq:wfhp}
\begin{split} 
\WFh'(T)\cap (W\times V)\defeq (W\times V)\setminus  \{ & (\rho', \rho)  \in  W\times V\; : \;  
\exists \, a \in S( T^* \RR^d ), \  b \in S ( T^* \RR^k ) \,, \\ 
& a ( \rho ) = 1 \,, b ( \rho' ) = 1 \,, \quad
 b^w  \,T\, a^w  = \cO_{L^2\to L^2} ( h^\infty ) \}\,. 
\end{split}
\end{equation}
If $ \WF_h' ( T )\cap (W\times V)  \Subset W \times V $,
there exists a family of tempered operators $\tT(h):L^2\to L^2$, such that $T$ and $\tT$ are
microlocally equivalent on $V$, while
$\tT$ is $ \cO_{ \cS' \rightarrow \cS} ( h^\infty ) $ outside $V$, that is 
$$
\tT\circ a^w=\cO(h^\infty):\cS'(\IR^d)\to\cS(\IR^k)\,,
$$ 
for all $a\in S(T^*\IR^d)$ such that
$\supp a\cap V=\emptyset$. This family, which 
is unique modulo $\cO_{ \cS' \rightarrow \cS}(h^\infty)$, is an
extension of the microlocally defined $T(h)$, see \cite[Chapter 10]{EZ}.

\subsection{Local $h$-Fourier integral operators.}
\label{fio}

We first present a 
a class of globally defined $ h$-Fourier integral operators 
following  \cite{SjZw02} and \cite[Chapter 10]{EZ}. 
This global definition will then be used to define
Fourier integral operators microlocally.

Let $ (A ( t ))_{t\in[-1,1]} $ be a smooth 
family of selfadjoint pseudodifferential
operators, 
$$\forall t\in [-1,1],\quad  A ( t ) = \Op ( a ( t ) ) \,, \ \  a (t)  
\in S_{cl} ( T^* \RR^d; \RR )   \,, $$ 
where the dependence on $ t $ is smooth, and 
$ \WFh ( A ( t )) \subset \Omega \Subset T^* \RR^d $, in the 
sense of \eqref{eq:WFA}. We then define
a family of operators
\begin{gather}
\label{eq:2.2}
\begin{gathered}
  U( t ) \; : \; L^2 ( \RR^d) \rightarrow L^2 ( \RR^d) \,, \\
h D_t U ( t) +  U ( t) A(t) = 0 \,. \ \ U( 0 ) = Id\,. 
\end{gathered}
\end{gather}
An example is given by $ A ( t) = A=a^w $, independent of 
$ t $,  in which case $ U ( t) = \exp ( - i t A / h ) $.

The family $ (U ( t ))_{t\in [-1,1]} $ is an example of a family of 
unitary $h$-Fourier integral operators, 
associated to the family of canonical transformations $ \kappa(t)$ 
generated by the (time-dependent) Hamilton 
vector fields $ H_{a_0(t)} $. Here the real valued function
$ a_0 (t) $ is the principal 
symbol of $ A ( t ) $ (see \eqref{e:classical-type}), and the
canonical transformations $\kappa(t)$ are defined through
\[ \frac{d}{dt} \kappa( t) (\rho) =  ( \kappa ( t ))_* ( H_{a_0 ( t)} 
(\rho )) \,, \ \ \kappa( 0 ) (\rho ) = \rho \,,   \ \ 
\rho \in T^* \RR^d \,.\]

If $ U = U(1) $, say, and the graph of $ \kappa (1) $ is denoted by $C$,
we conform to the usual notation and write
\[  U \in I^0_h ( \RR^d \times \RR^d ; C' ) \,, \quad \text{where}\quad
C' = \{ ( x, \xi;  y , - \eta) \; : \; ( x, \xi) = \kappa ( y , \eta )   \}
\,.  \]
Here the twisted graph $C'$ is a Lagrangian submanifold of $T^*(\RR^d\times\RR^d)$.

In words, $ U $ is a unitary $h$-Fourier integral operator associated to 
the canonical graph $ C$ (or the symplectomorphism $\kappa(1)$ defined by this graph). 
Locally all unitary $h$-Fourier integral operators
associated to canonical graphs are of the form $ U ( 1) $, 
since each local canonical transformation with a fixed point can be 
deformed to the identity, see \cite[Lemma 3.2]{SjZw02}. 
For any $\chi\in S(T^*\RR^d)$,
the operator $U(1)\,\chi^w$,
with $\chi\in S(T^*\RR^d)$ is still a (nonunitary) $h$-Fourier integral operator
associated with $C$. The class formed by these operators, which are
said to
``quantize'' the symplectomorphism $ \kappa = \kappa( 1 )$, depends only on 
$ \kappa $, and not on the deformation path from the identity 
to $ \kappa $. This can be seen from the Egorov characterization of
Fourier integral operators -- see
\cite[Lemma 2]{SjZw02} or \cite[\S 10.2]{EZ}. 

Let us assume that a symplectomorphism $\kappa$ is defined only near the
origin, which is a fixed point. It is always possible to {\em locally} deform
$\kappa$ to the identity, that is construct a family of symplectomorphisms 
$\kappa(t)$ on $T^*\RR^d$, such that $\kappa(1)$ coincides with
$\kappa$ in some neighbourhood $V$ of the origin \cite[Lemma
3.2]{SjZw02}. If we apply the above construction to get the unitary operator
$U(1)$, and use a cutoff $\chi\in S(T^*\RR^d)$,
$\supp\chi\Subset V$, then the operator $U(1)\chi^w$ is an $h$-Fourier
integral operator associated with the local symplectomorphism
$\kappa\rest V$. Furthermore, if there exists a neighbourhood $V'\Subset V$ such that
$\chi\rest V'\equiv 1$, then
$U(1)\chi^w$ is microlocally unitary inside $V'$.

For an open set $ V \Subset \RR^d $ and $ \kappa $
a symplectomorphism defined in a neighbourhood $\tV$ of $ V $, 
we say that a tempered operator $ T $ satisfying
\[ 
T \; : \; H ( \tV ) \longrightarrow H ( \kappa ( \tV ) ) \,, 
\]
is a micrololocally defined {\em unitary}
$h$-Fourier integral operator in $V$, if any point $ \rho \in V $ has 
a neighbourhood $ V_\rho \subset V $ such that 
\[ 
T \; : \; H ( V_\rho ) \longrightarrow H ( \kappa ( V_\rho ) ) 
\]
is equivalent to a unitary $h$-Fourier integral operator associated
with $\kappa\rest V_\rho$, as defined by the
above procedure.
The microlocally defined operators can also be obtained by 
oscillatory integral constructions --- see for instance 
\cite[\S 4.1]{NZ2} for a brief self-contained presentation.

\medskip

An example which will be used in \S \ref{ml} is given 
by the standard conjugation result, see \cite[Proposition 3.5]{SjZw02}
or \cite[Chapter 10]{EZ} for self-contained proofs. 
Suppose that $ P \in \Psi_{cl}^{m,0} ( \RR^d ) $ is a semi-classical
{\em real principal type operator}, namely its principal symbol $ p = \sigma_h ( P ) $ is real,
independent of $ h$, and the Hamilton flow it generates has no fixed
point at energy zero:
$ p = 0 \Longrightarrow dp \neq 0 $.
Then for any $ \rho_0 \in p^{-1}(0) $, 
there exists a canonical transformation, 
$ \kappa $, mapping $ V = \neigh((0,0),T^*\RR^d)$ to $\kappa(V)=\neigh( \rho_0, T^* \RR^d) $, with 
$\kappa( 0 , 0 )=\rho_0 $ and
$$
p\circ\kappa(\rho)=\xi_n(\rho) \quad\rho\in V\,,
$$ 
and a unitary microlocal $ h$-Fourier integral operator
$U : H( V ) \rightarrow H ( \kappa ( V ) )$
associated to $ \kappa $, such that
\begin{gather*}
U^* P U  \equiv hD_{x_n} : H (V ) \rightarrow H ( V ) \,. 
\end{gather*}
While $\xi_n$ is the (classical) normal form for the
Hamiltonian $p$ in $V$, the operator $hD_{x_n}$ is the quantum normal form for $P$,
microlocally in $V$.

The definition of $h$-Fourier integral operators 
can be generalized to graphs $C$ associated with certain {\em relations} 
between phase spaces of possibly different dimensions. 
Namely, if a relation $C\subset T^*\RR^d\times T^*\RR^{k}$
is such that its twist
$$
C'=\{(x,\xi;y,-\eta)\, ;\, \ (x,\xi;y,-\eta')\in C\} 
$$ 
is a Lagrangian submanifold of $ T^*(\RR^d\times \RR^k) $, 
then one can associate with this relation (microlocally in some neighbourhood) 
a family of $h$-Fourier integral operators $T:L^2(\RR^k)\mapsto L^2(\RR^d)$  
\cite[Definition 4.2]{Al}. This class of operators is denoted by $I^r_h ( \RR^d \times \RR^{k} ; C' )$, with 
$r\in\RR$. The important property of these operators is that their composition is still a
Fourier integral operator associated with the composed relations.

\subsection{Complex scaling}
\label{cs}

We briefly
recall the complex scaling method of Aguilar-Combes \cite{AgCo} 
-- see \cite{SjZw91},\cite{Sj}, and references
given there. In most of this section, this scaling is independent of
$h$, and allows to obtain the resonances (in a certain sector) for all operators $P(h)$, 
$ h\in (0,1] $, where
$ P(h) $ satisfies the assumptions of \S\ref{ass}.

For any $ 0 \leq \theta \leq \theta_0$ and $R>0$, we define $ \Gamma_{\theta,R}
\subset \CC^n $ to be a totally real deformation of 
$ \RR^n $, with the following properties:
\begin{gather}
\label{eq:gpr}
\begin{gathered}
\Gamma_\theta \cap B_{\CC^n } ( 0 , R ) = B_{\RR^n } ( 0 , R ) \,, \\
\Gamma_\theta \cap \CC^n \setminus B_{\CC^n } ( 0 , 2 R ) =
e^{ i \theta } \RR^n \cap \CC^n \setminus B_{\CC^n } ( 0 , 2 R ) \,,\\
\Gamma_\theta = \{ x + i f_{\theta, R} ( x) \; : \; x \in \RR^n \} \,,  \  \
\partial_x^{\alpha}  f_{\theta, R} ( x ) =
\Oo_{\alpha} ( \theta )  \,.
\end{gathered}
\end{gather}
If $ R $ is large enough, 
the coefficients of $ P$ continue analytically outside of
$ B ( 0 , R ) $, and we can define a dilated operator:
\[ 
P_{\theta, R} \defeq \widetilde P\rest_{ \Gamma_{\theta, R } } \,,
\ \  P_{\theta, R} u  = \widetilde P ( \tilde u ) \rest_{\Gamma_{\theta, R } } \,,
\]
where $ \widetilde P $ is the holomorphic continuation of the
operator $ P $, and $ \tilde u$ is an almost analytic extension of
$ u \in \CIc ( \Gamma_{\theta, R } ) $ from the totally real submanifold
$ \Gamma_{\theta, R } $ to $ \neigh(\Gamma_{\theta, R }, \CC^n) $.

The operator $ P_{\theta , R } - z $ is a Fredholm operator for 
$ 2\theta >\arg ( z + 1 ) > - 2 \theta $. That means that
the resolvent, 
$ ( P_{\theta, R} - z )^{-1} $, is meromorphic 
in that region,
the spectrum of $ P_{\theta, R} $ in that region
is independent of $ \theta $ and $ R$, and 
consists of the {\em quantum resonances} of $ P$. 

To simplify notations we identify
$ \Gamma_{\theta, R } $ with $ \RR^n $ using the map, $ S_{\theta, R } 
 : \Gamma_{\theta, R } \rightarrow \RR^n $, 
\be\label{eq:idt}
\Gamma_{\theta, R }  \ni x \longmapsto \Re x \in \RR^n \,,
\ee
and using this identification, consider $ P_{\theta, R} $ as an operator on $ \RR^n $,
defined by $ (S_{\theta, R}^{-1})^* P_{\theta, R} S_{\theta, R}^* $ (here $S^*$ means the pullback through $S$)
We note that this identificaton satisfies
\[ 
C^{-1}\,\| u  \|_{ L^2 ( \RR^n) } 
\leq \| S_{\theta, R }^* u  \|_{ L^2 ( \Gamma_{\theta, R} ) } \leq C\, \| u  \|_{L^2 (\RR^n ) }\,,
\]
with $ C $ independent of $ \theta $ if $ 0 \leq \theta \leq \theta_0$.

The identification of the eigenvalues of $ P_{\theta, R} $ with the 
poles of the meromorphic continuation of
$$
( P - z )^{-1} \; : \; \CIc ( \RR^n ) \; \longrightarrow \; \CI ( \RR^n )
$$
from $ \{\Im z>0\}$ to  $ D(0,\sin(2\theta)) $, and in fact, the 
existence of such a continuation, follows from the following 
formula (implicit in \cite{Sj}, and discussed in \cite{TaZw}):
if $ \chi \in \CIc ( \RR^n ) $, $ \supp \chi \Subset B ( 0 , R ) $, 
then 
\begin{equation}
\label{eq:two_r}
  \chi ( P_{\theta, R } - z )^{-1} \chi = \chi ( P - z)^{-1} \chi \,.
\end{equation}
This is initially valid for $ \Im z > 0 $ so that the right hand side
is well defined, and then by analytic continuation in the region 
where the left hand side is meromorphic.
\begin{figure}
\begin{center}
\includegraphics[width=0.6\textwidth]{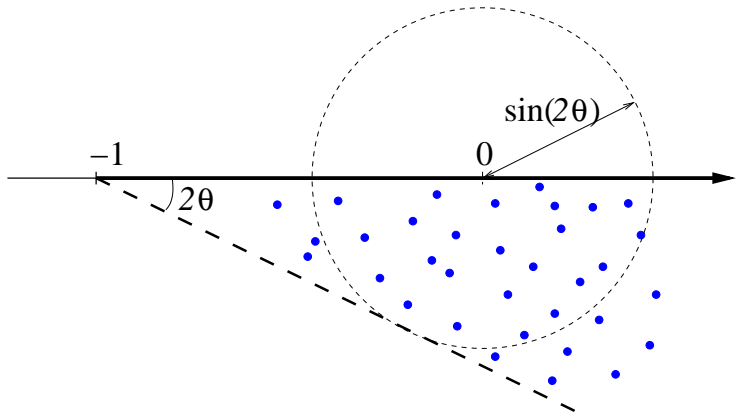}
\caption{The complex scaling in the $ z$-plane used in this paper.}
\label{f:com}
\end{center}
\end{figure}
The reason for the Fredholm property of $ (P_{\theta, R } - z )$ in $D(0,\sin(2\theta))$ comes from the properties of 
the principal symbol of $ P_{\theta, R} $ -- see Fig.~\ref{f:com}. 
Here for convenience,
and for applications to our setting, 
we consider $P_{\theta, R }$ as an operator 
on $ L^2 ( \RR^n )$ using the identification $S_{\theta,R}$ above. 
Its principal symbol is given by 
\begin{equation}
\label{eq:pth}
p_{\theta, R} ( x, \xi ) = p ( x + i f_{\theta, R} ( x) , [( 1 + i df_{\theta, R} (x))^t]
^{-1}
\xi ) \,, \quad (x,\xi)\in T^*\RR^d\,,
\end{equation}
where the complex arguments are allowed due to the analyticity of 
$ p ( x, \xi ) $ outside of a compact set --- see \S \ref{ass}.
We have the following properties
\begin{equation}
\label{eq:pth1}
\begin{split}
& \Re p_{\theta, R} ( x , \xi) = p ( x, \xi ) + {\mathcal O} ( \theta^2) 
\langle \xi \rangle^2 
 \,, \\ 
& \Im p_{\theta, R} ( x, \xi ) = -  d_\xi p ( x, \xi ) [ df_{\theta, R} ( x )^t \xi ]
+ d_x p (x , \xi ) [ f_{\theta, R} ( x ) ] + \Oo ( \theta^2) 
\langle \xi \rangle^2  \,.
\end{split}
\end{equation}
This implies, for $R$ large enough,
\be\label{eq:pth3}
| p (x,\xi) | \leq \delta \,, \ \ 
|x | \geq 2\,R \ \Longrightarrow \ \Im p_{\theta, R} ( x, \xi) \leq - C \theta \,.
\ee

For our future aims, it will prove convenient to actually let the
angle $ \theta $ explicitly depend on $ h $: as long as
$ \theta > c h \log ( 1/ h ) $, the estimates above guarantee 
the Fredholm property of $(P_{\theta, R } - z )$ for 
$ z \in D ( 0 , \theta / C ) $, by providing approximate inverses
near infinity. We will indeed take $\theta$ of the order of
$h\log(1/h)$, see \eqref{e:theta}.

\subsection{Microlocally deformed spaces}
\label{cef}
Microlocal deformations using exponential weights have played an important 
role in the theory of resonances since \cite{HeSj}. Here we take an 
intermediate point of view \cite{Ma,SjZw04} 
by combining compactly supported weights with the complex scaling
described above. 
We should stress however that the full power of \cite{HeSj} would allow
more general behaviours of $p(x,\xi)$ at infinity, for instance potentials
growing in some directions at infinity.

Let us consider an $h$-independent real valued
function $G_0 \in \CIc ( T^* \RR^d; \RR )$, and rescale it in an
$h$-dependent way:
\begin{equation}
\label{eq:G0}  G ( x , \xi ) = M h \log (1/h ) G_0 ( x , \xi )
\,,\quad M>0\ \text{fixed}.
\end{equation}
For $ A \in \Psi^{m,0} ( \RR^d ) $, we consider the conjugated operator
\begin{equation}
\label{eq:exGe}
\begin{split} 
e^{-G^w ( x, h  D) / h } A e^{ G^w ( x, h D )/h } & = e^{-\ad_{G^w ( x, h D) }/h } 
A \\
& = \sum_{\ell=0}^{L-1} \frac{(-1)^\ell} { \ell!} \left( 
\frac 1 h  \ad_{G^w ( x, h D ) }\right) ^\ell A  + R_L \,,
\end{split}
\end{equation}
where
\[ R_L = 
\frac{ (-1)^L } { L!} \int_0^1 e^{ - t G^w ( x, h D ) } \left( \frac 1 h
\ad_{ G^w ( x, h D ) } \right) ^{L} A e^{t G^w ( x , h D ) } dt \,. \]
The semiclassical calculus of pseudodifferential operators
\cite[Chapter 7]{DiSj},\cite[Chapter 4, Appendix D.2]{EZ} 
and \eqref{eq:G0} show that 
\[  \left( 
\frac 1 h  \ad_{G^w ( x, h D ) }\right) ^\ell A = 
( M\log(1/h))^\ell \, (\ad_{G_0^w ( x, h D ) }) ^\ell A
\in (M h \log(1/h))^\ell \,
\Psi_{h}^{-\infty, 0} ( \RR^d ) \,, \ \ \forall\ell > 0 \,.   \]
Since $\|G_0^w\|_{L^2\to L^2}\leq C_0$, functional calculus of 
bounded self-adjoint operators shows that
$$ \| \exp ( \pm t G^w ( x, h D ) ) \| \leq h^{-tC_0M} \,,
$$  so 
we obtain the bound, 
\[ R_L =  \cO_{L^2\rightarrow L^2}  ( \log(1/h)^L\,h^{L-2tC_0M} ) =
\cO_{L^2\rightarrow L^2}  ( h^{L-2tC_0M-L\delta} ) \,,
\]
with $\delta>0$ arbitrary small. 
Applying this bound,
we may write \eqref{eq:exGe} as 
\begin{equation}
\label{eq:exGex} 
e^{-G^w ( x, h  D) / h } A e^{ G^w ( x, h D ) /h } \sim 
 \sum_{\ell=0}^\infty  \frac{(-1)^\ell} { \ell!} \left( 
\frac 1 h  \ad_{G^w ( x, h D ) }\right) ^\ell A \in \Psi^{m,0}(\IR^d)\,.
\end{equation}
In turn, this expansion, combined with Beals's characterization of
pseudodifferential operators \eqref{eq:beals}, implies that the
exponentiated weight is a pseudodifferential operator:
\begin{equation}
\label{eq:exG}
  \exp ( G^w ( x, h D ) / h ) \in \Psi^{0,C_0 M}_{\delta} ( \RR^d ) \,, \ \ \forall 
\delta > 0 \,. 
\end{equation}
Using the weight function $G$, we can now define our weighted spaces.
Let $ H_h^k ( \RR^d ) $ be the semiclassical 
Sobolev spaces defined in \eqref{eq:Sob}. We put
\be\label{e:deformed-norms}  
H_G^k ( \RR^d ) = e^{ G^w ( x, h D )/h } H_h^k ( \RR^d ) \,, \ \ 
\| u \|_{H_G^k } \defeq \| e^{- G^w ( x, h D )/h } u \|_{ H^k_h } \,, 
\ee
and 
\[ \langle  u , v \rangle_{H_G^k } = \langle e^{- G^w ( x, h D )/h } u , 
e^{- G^w ( x, h D )/h } v \rangle_{H^k_h }\,. \]
As a vector space, $H_G^k ( \RR^d )$ is identical with $H_h^k ( \RR^d )$, but the
Hilbert norms are different. In the case of $ L^2 $, that is of $ k = 0 $, we simply 
put $ H_G^0 = H_G $. 


The mapping properties of
$ P = p^w ( x, h D )  $ on $ H_G ( \RR^d ) $ are equivalent with those
of $P_G\defeq e^{-G^w/h}P\,e^{G^w/h}$ on $L^2(\RR^d)$, which are 
governed by the
properties of the (full) symbol $p_G$ of $P_G$: formula \eqref{eq:exGex} shows that
\begin{equation}
\label{eq:pG}
 p_G = p - i H_p G + \cO ( h^2 \log^2 ( 1/h ) ) \,.
\end{equation}

\medskip

At this moment it is convenient to introduce a notion of
{\em leading symbol}, which is adapted to the study of conjugated operators such as $P_G$.
For a given $ Q \in S ( T^* \RR^d ) $, we say that $ q \in S ( T^* \RR^{d} )$ is a leading 
symbol of $ Q^w ( x, h D ) $, if
\begin{equation}
\label{eq:leads}
\forall \gamma \in (0,1) \,, \forall\alpha,\,\beta\in \NN^{d}\,,  
\quad h^{-\gamma } \partial_x^{\alpha }\partial^\beta_\xi (Q-q) = \cO_{\alpha,\beta} ( \la \xi\ra^{-|\beta|} ) \,, 
\end{equation}
that is, $ ( Q - q ) \in S^{0,-\gamma}( T^* \RR^d ) $ for any $\gamma \in ( 0 , 1 ) $. This property
is obviously an equivalence relation inside $S ( T^* \RR^{d} )$, which is weaker than the
equivalence relation defining the principal symbol map on
$\Psi_h$ (see \S\ref{sa}). 
In particular, terms of the size $h\log(1/h)$ are ``invisible'' to the
leading symbol.
For example, the leading symbols of $p_G$ and $p$ are the same.
If we can find $ q $ independent of $h$, then it is unique.

For future use we record the following:
\begin{lem}
\label{l:saG}
Suppose 
\[ Q^w ( x , h D ) \; : \; H_G ( \RR^d ) \longrightarrow H_G ( \RR^d ) \,,
\ \ Q \in S ( T^* \RR^d ) \,, \]
is self-adjoint (with respect to the Hilbert norm on $H_G$). 
Then this operator admits a {\it real} leading symbol. 
Conversely, if $ q \in S ( T^* \RR^d ) $ is real, then 
there exists $ Q \in S ( T^* \RR^d ) $ with leading symbol
$ q $, such that $ Q^w ( x, h D) $ is self-adjoint on $ H_G ( \RR^d ) $.
\end{lem}
\begin{proof}
This follows from noting that 
$$ Q^w_G \defeq e^{-G^w/h}Q^w ( x , h D ) e^{G^w/h} \,, $$
has the same leading symbol as $ Q^w ( x , h D) $, and that 
self-adjointness of $ Q^w $ on $ H_G $ is equivalent to 
self-adjointness of $ Q_G^w $ on $ L^2 $: the definition of $H_G $ in 
\eqref{e:deformed-norms}
(the case of $ k = 0 $) gives
\[ \langle Q^w u , v \rangle_{H_G} = \langle 
e^{ - G^w /h }  Q^w u , e^{- G^w  / h } v \rangle_{L^2} 
=  \langle 
Q_G^w ( e^{ -G^w /h }u ), e^{- G^w / h } v \rangle_{L^2} \,. \]
\end{proof}

The weighted spaces can also be microlocalized in the 
sense of \S\ref{mic}: for 
$V   \Subset T^* \RR^d$, we define the space
\begin{equation}\label{eq:HVG}
\begin{split} 
 H_G( V ) \defeq  
\{ u = u ( h )\in H_G(\RR^d),  \; : \; &\exists  C_u  > 0 \,,\ \forall
h\in (0,1],\ \| u ( h ) \|_{H_G ( \RR^d) } \leq C_u\\
& \exists \, \chi \in \CIc (V) 
\,, \quad 
\chi^w u =  u +\cO_{\cS}  ( h^\infty )
 \} \,. 
\end{split}
\end{equation}
In other words, $H_G ( V ) = e^{G^w ( x, h D) / h } H( V )$.
This definition depends only on the values of the weight $ G$ in the
open set $ V $.

For future reference we state the following 
\begin{lem}
\label{l:TGV}
Suppose $ T : H ( V ) \rightarrow H ( \kappa ( V ) ) $ is an 
$ h $-Fourier integral operator associated to a symplectomorphism
$ \kappa $ (in the sense of \S \ref{fio}), and is  asymptotically
uniformly bounded (in the sense of \eqref{e:asymp-bounded}). 
Take
$ G_0 \in \CIc ( \neigh(\kappa( V ))) $, $G = M h \log (1/h) G_0$.\\ Then the operator
\begin{equation}
\label{eq:Tbd}
T \; : \; H_{\kappa^*G } ( V ) \to H_G ( \kappa ( V ) ) 
\end{equation}
is also asymptotically uniformly bounded with respect to the deformed norms.
\end{lem}
\begin{proof}
Since the statement
is microlocal we can assume that $ V $ is small enough, so that
$ T \equiv T_0 A $ in $V$, where $ T_0 $ is unitary on $ L^2 ( \RR^d ) $ and 
$ A \in \Psi_h $. As in the proof of Lemma \ref{l:saG} the boundedness
of \eqref{eq:Tbd} is equivalent to considering the boundedness of 
\[ e^{-G^w ( x ,h D ) } T_0 e^{ ( \kappa^* G)^w ( x , h D ) /h } 
A_{ \kappa^* G } \; : \; L^2 ( \RR^d )  \rightarrow L^2 ( \RR^d ) \,, \]
where 
\[ A_{\kappa^* G } \defeq  e^{ - ( \kappa^* G)^w ( x , h D ) /h } A 
 e^{ ( \kappa^* G)^w ( x , h D ) /h } \,. \]
Because of \eqref{eq:exGex}, we have uniform boundedness of 
$  A_{\kappa^* G } $ on $ L^2 $. Unitarity of $ T_0 $ means that it is 
sufficient to show the uniform boundedness of 
\[\begin{split} T_0^{-1}  e^{-G^w ( x, h D) / h }  T_0 
e^{ ( \kappa^* G)^w ( x, h D) /h }
& = e^{- M\log(1/h) (T_0^{-1}  G_0^w ( x, h D)   T_0 ) }  
e^{ M\log(1/h)( \kappa^* G_0)^w ( x, h D) }
\end{split} \]
on $ L^2 $. 
Egorov's theorem (see \cite[\S 10.2]{EZ}) shows that
\[ T_0^{-1}  G_0^w ( x, h D)   T_0 =  G_\kappa ( x, h D ) \,, \quad
G_\kappa - \kappa^* G_0 \in \Psi^{-\infty,-1}_h(\RR^d)\,.\]  
Since $ [ G_\kappa^w , \kappa^* G_0^2 ] = h^2 B $, $ B \in \Psi^{-\infty,0}_h(\RR^d) $, 
the Baker-Campbell-Hausdorff formula for bounded operators 
shows\footnote{Alternatively, we can compare $ \exp ( M \log ( 1/h) G_\kappa^w ) $ 
with $ \left(\exp ( M \log ( 1/h ) G_\kappa ) \right)^w $ and use product
formul{\ae} for pseudodifferential operators -- see \cite[Appendix]{SjZw04}
or \cite[Section 8.2]{EZ}.}
that 
\[
\begin{split} T_0^{-1}  e^{-G^w ( x, h D) / h }  T_0 
e^{ ( \kappa^* G)^w ( x, h D) /h }
&= e^{-  M\log(1/h) G_\kappa^w ( x, h D ) } \,e^{ M\log(1/h)( \kappa^* G_0)^w ( x, h D) } \\
& = e^{  M\log(1/h) ( - G_\kappa^w ( x, h D ) +\kappa^* G_0)^w ( x, h D) ) 
+ \Oo_{L^2 \rightarrow L^2} ( \log ( 1/ h )^2 h^2 ) } \\
& = \exp{  \Oo_{L^2 \rightarrow L^2 }  ( h \log ( 1/h ) ) } \\
& = 
\Id + \Oo_{L^2 \rightarrow L^2 } ( h \log ( 1/h ) ) \,. 
\end{split} \]
This proves uniform bounded of globally defined operators $ T_0 A $, and 
the asymptotic uniformly boundedness 
in the sense of \eqref{e:asymp-bounded} of $ T $ on spaces of 
microlocally localized functions.
\end{proof}

\subsection{Escape function away from the trapped set}
\label{eft}

In this section we recall the construction of the specific weight function $G$ which,
up to some further small modifications, will be used to prove Theorems~\ref{t:s} and \ref{t:mg}. 

Let $ K_E \subset p^{-1} ( E ) $ be the trapped
set on the $ E $-energy surface, see \eqref{eq:K0}, and define
\begin{equation}
\label{eq:Khd} \whK  =\whK_\delta \defeq \bigcup_{ | E | \leq \delta } 
K_E \,. 
\end{equation}
The construction of the weight function is based on the following 
result of \cite[Appendix]{GeSj}:
for any open neighbourhoods $ U, V $
of $ \whK$, $  \overline U \subset V $, there exists 
$ G_1 \in \CI ( T^* X ) $, such that
\begin{equation}
\label{eq:gsa0}
G_1\rest_U\equiv 0\,, \  \ H_p G_1 \geq 0 \,, \ \
H_p G_1 \rest_{ p^{-1}( [-2\delta, 2 \delta ])   } \leq C \,,\ \ 
H_p G_1 \rest_{  p^{-1} ([- \delta, \delta ]  )\setminus V } \geq 1 \,.
\end{equation}
These properties mean that $ G_1 $ is an {\em escape function}: it increases along the flow,
and {\em strictly increases} along the flow 
on $ p^{-1}( [-\delta, \delta ] ) $ away from $ \whK$
(as specified by the neighbourhood $ V $). Furthermore, $ H_p G $ is bounded in 
a neighbourhood of $ p^{-1} ( 0 ) $. 

Since such a function $G_1$ is necessarily of unbounded support, we need to 
modify it to be able to use $ H_G$-norms defined in \S \ref{cef}
(otherwise methods of \cite{HeSj} could be used and that alternative
would allow more general behaviours
at infinity, for instance a wide class of polynomial potentials).  
For that we follow \cite[\S\S 4.1,4.2,7.3]{SjZw04} and \cite[\S 6.1]{NZ2}:
$ G_1 $ is modified to a compactly supported $ G_2 $ 
in a way allowing complex scaling estimates \eqref{eq:pth3} to 
compensate for the wrong sign of $ H_p G_2 $.
Specifically, \cite[Lemma 6.1]{NZ2} states that
for any large $ R>0 $ and $ \delta_0 \in ( 0,1/2) $ we can construct $ G_2 $ with the following
properties:
$  G_2  \in \CIc ( T^*X ) $ and 
\be
\label{eq:gsa} 
\begin{split}
& H_p G_2  \geq 0 \ \ \ \ \ \  \ \  \text{on $ {T^*_{B(0,3R )}  X }$,}  \\ 
& H_p G_2 \geq 1 \ \ \ \ \ \ \; \, \text{ on 
$ { T^*_{B(0,3R )} X \cap (p^{-1}([-\delta, \delta ] ) \setminus V) }$, }
\\
& H_p G_2   \geq -\delta_0  \ \ \ \ \text{ on $ T^* X $.} 
\end{split}
\ee
Let 
$$ 
G \defeq M h \log ( 1/ h ) G_2 \,, \quad\text{with $M>0$ a fixed constant.}
$$ 
Then, in the notations of \S\ref{cef}, we will be interested in the
complex-scaled operator
\[  P_{\theta, R} \; : \; H_G^2 ( \RR^n ) \longrightarrow H_G ( \RR^n
) \,,\]
for a scaling angle depending on $h$:
\be\label{e:theta(h)}
\theta=\theta(h)= M_1\,h\log(1/h),\quad M_1>0\ \text{fixed}.
\ee
 Inserting the above estimates in (\ref{eq:pG}), we get
\begin{equation}
\label{eq:estaway} 
| \Re p_{\theta, R ,G } ( \rho ) | < \delta/2 \,, \ \ \Re \rho \notin 
V \,, \  \Longrightarrow \ 
\Im p_{\theta, R,G } ( \rho ) \leq - \theta/C_1 \,, 
\end{equation}
provided that we choose \cite[\S 6.1]{NZ2}
\be\label{e:theta}
\frac{M}{C}\geq M_1 \geq \frac{\delta_0 M}{C},\qquad\text{for some }C>0\,,
\ee
Assuming that the constant $M_0$ appearing in the statement of
Theorem~\ref{t:s} satisfies
$$0 < M_0 \leq M_1\,$$ 
for $\delta>0$ and $h>0$
small enough, the rectangle $\DOCh$ is contained in the uncovered
region in Fig.~\ref{f:com}, hence the scaling by the angle
\eqref{e:theta(h)} gives us access to the resonance spectrum in
the rectangle $\DOCh$. In \S\ref{s:final} we will need to further adjust
$M_0$ with respect to $M_1$.


\renewcommand\thefootnote{\ddag}%

\subsection{Grushin problems}
\label{gp}
In this section we recall some linear algebra facts related to the 
Schur complement formula, which are at the origin of the Grushin
method we will use to analyze the operator $P_{\theta,R}$. 

For any invertible square matrix decomposed into $4$ blocks, we have
\[ \begin{pmatrix} p_{11} &p_{12}\\ p_{21} &p_{22}\end{pmatrix}^{-1}  = 
 \begin{pmatrix} q_{11} &q_{12}\\ q_{21} &q_{22}\end{pmatrix} 
\ \Longrightarrow \ p_{11}^{-1} = q_{11} - q_{12} \, q_{22}^{-1} q_{21}\,,
\]
provided that $ q_{22}^{-1} $ exists (which implies that 
$ q_{22} $, and hence $ p_{11} $, are square matrices). We have the 
analogous formula for $ q_{22}^{-1} $: 
\[ q_{22}^{-1} = p_{22}-p_{21}p_{11}^{-1}p_{12} \,. \]

One way to see these simple facts is to apply 
gaussian elimination to 
\[ 
{\mathcal P} =\begin{pmatrix} p_{11} &p_{12}\\ p_{21} &p_{22}\end{pmatrix}
\] so that, if $p_{11}$ is invertible, we have an 
upper-lower triangular factorization:
\begin{equation}
\label{eq:uplo}
{\mathcal P}=\begin{pmatrix}p_{11} & 0\\p_{21} &1
\end{pmatrix}\begin{pmatrix} 1 &p_{11}^{-1}p_{12}\\ 0 &p_{22}-p_{21}p_{11}^{-1}p_{12}\end{pmatrix} \,. 
\end{equation} 

The formula for the inverse of $ p_{11} $
 leads to the construction of {\em effective Hamiltonians} for
operators (quantum Hamiltonians) $ P:\cH_1\to\cH_2 $. We first search for auxiliary 
spaces $ {\mathcal H}_\pm $ and  operators $ R_\pm $ for which the matrix of operators
\[   \begin{pmatrix} P - z & R_- \\
R_+ &  \ 0 \end{pmatrix} \; : \; {\mathcal H}_1 \oplus {\mathcal H}_- 
\longrightarrow {\mathcal H}_2 \oplus {\mathcal H}_+ \,,\]
is invertible for $z$ running in some domain of $\CC$. Such a matrix is called a {\em Grushin problem}, and
when invertible the problem is said to be {\em well posed}.

When successful this procedure reduces the spectral problem for $ P $ to a nonlinear
spectral problem of lower dimension. 
Indeed, if $  \dim \cH_- = \dim \cH_+ < \infty $, we write
\[  
\begin{pmatrix}  P - z  & R_- \\
R_+ &  0 \end{pmatrix}^{-1} = 
 \begin{pmatrix} E ( z ) & E_+ ( z ) \\
E_- (z )  &  E_{-+} ( z )  \end{pmatrix}\,, 
\]
and the invertibility of $ (P - z):\cH_1\to\cH_2 $ is equivalent to the invertibility of the
finite dimensional matrix $ E_{-+} ( z ) $. The zeros of $ \det E_{-+} ( z ) $ coincide with the eigenvalues of 
$ P $ (even when $ P $ is not self-adjoint) because of the following
formula:
\begin{equation}
\label{eq:trace}
\tr \oint_z ( P - w )^{-1} dw = 
- \tr \oint_z E_{-+}(w) ^{-1} 
E_{-+}'(w)\, d w\,,
\end{equation}
valid when the integral on the left hand side is of trace
class -- see \cite[Proposition 4.1]{SZ9} or verify it using the 
factorization \eqref{eq:uplo}. Here $ \oint_z $ denotes
an integral over a small circle centered at $ z$.
The above formula shows that $\dim\ker(P-z)=\dim\ker E_{-+}(z)$.

The matrix $E_{-+} ( z )$ is often called 
an {\em effective Hamiltonian} for the original Hamiltonian $P$ -- see \cite{SZ9} for a review of this formalism
and many examples. In the physics literature, this reduction is
usually called the Feshbach method.


We illustrate the use of Grushin problems with a simple lemma which will be useful later in \S\ref{s:final}.
\begin{lem}
\label{l:indg}
Suppose that 
\[ 
\cP \defeq  \begin{pmatrix} P & R_- \\
R_+ &   0 \end{pmatrix} \; : \; \cH_1 \oplus \cH_- 
\longrightarrow \cH_2 \oplus \cH_+ \,,\]
where $ \cH_j $ and $ \cH_\pm $ are Banach spaces. 
If $ P^{-1} : \cH_2 \rightarrow \cH_1 $ exists then 
\[ 
\cP \ \text{is a Fredholm operator} \ 
\Longleftrightarrow 
\  R_+ P^{-1} R_-:\cH_-\to\cH_+  \ \text{is a Fredholm operator} \,, 
\]
and 
\[ \ind \cP = \ind { R_+ P^{-1} R_-} \,. \]
\end{lem}
\begin{proof} 
We apply the factorization \eqref{eq:uplo} with $ p_{11} = P $, 
$ p_{12} = R_- $, $ p_{21} = R_+ $, $ p_{22} = 0 $. Since the
first factor is invertible we only need to check the 
the Fredhold property and the index of the second factor:
\[  \begin{pmatrix} 1 & P^{-1} R_- \\
0 & - R_+ P^{-1} R_- \end{pmatrix} \,, \]
and the lemma is immediate.
\end{proof}

\section{A microlocal Grushin problem}
\label{gps}

In this section we recall and extend the analysis of \cite{SjZw02} to 
treat a Poincar\'e section $\bSigma\subset p^{-1}(0)$ for
a flow satisfying the assumptions in \S \ref{da}. 
In \cite{SjZw02} a Poincar\'e section associated to a single closed orbit
was considered. The results presented here are purely microlocal
in the sense of \S\ref{mic}, first near a given component $\Sigma_k$ of the 
section, then near the trapped set $ K_0 $.
In this section $P$ is the original differential operator, but it
could be replaced by its complex scaled version $ P_{\theta, R } $, since the complex deformation
described in \S\ref{cs} takes place far away from $ K_0$. Also, 
when no confusion is likely to occur, we will often denote the Weyl quantization $\chi^w$
of a symbol $ \chi\in S(T^*\RR^d) $ by the same letter: $ \chi = \chi^w $.

\subsection{Microlocal study near $ \Sigma_k $.}
\label{ml}

First we focus on a single component $\Sigma_k$ of the Poincar\'e section,
for some arbitrary $k\in\{ 1,...,N\}$. Most of the time we will then
drop the subscript $k$. 
Our aim is to construct a microlocal Grushin problem for the operator 
\[ \frac i h (P-z) \,, \]
near $\Sigma=\Sigma_k$, where $|\Re z|\leq\delta$, $|\Im z|\leq M_0 h
\log( 1/h)$, and $\delta$ will be chosen small enough
so that the flow on $\Phi^t\rest_{K_{\Re z}}$ is a small perturbation of $\Phi^t\rest_{K_0}$.

\subsubsection{A normal form near $\Sigma_k$}\label{s:normal}
Using the assumption \eqref{eq:H5} and a version of Darboux's theorem (see for instance 
\cite[Theorem 21.2.3]{Hor2}), we may extend the map $\kappa_k =\kappa:\wSi_k\to\Sigma_k $ 
to a canonical transformation $\tkappa_k$ defined in a neighbourhood of $\wSi_k$ in $T^*\IR^n$,
$$
\tOmega_k\defeq 
\{ (x,\xi)\in T^*\IR^n; (x',\xi ')\in \wSi_k,\,\ |x_n|\leq\eps,\ |\xi_n|\leq\delta\}\,,
$$ 
such that
\be
\tkappa_k(x',0,\xi ',0)=\kappa_k(x',\xi ')\in\Sigma_k \,, \qquad p\circ \tkappa_k =\xi_n \,.
\ee
We call $\Omega_k=\tkappa_k(\tOmega_k)$ the neighbourhood of $\Sigma_k$ in $T^*X$ in the range of $\tkappa_k$.
The ``width along the flow'' $\eps>0$ is taken small enough, so that the sets $\{\Omega_k,\,k=1,\ldots,N\}$
are mutually disjoint, and it takes at least a time $20\eps$ for a point to travel between any $\Omega_k$
and its successors.

The symplectic maps $\tkappa_k$ allow us to extend the Poincar\'e
section $\bSigma$ to the
neighbouring energy layers $p^{-1}(z)$, $z\in [-\delta,\delta]$. Let us call
$$
\kappa_{k,z}\defeq \tkappa_k\rest (\tOmega_k\cap \{\xi_n=z\})\,.
$$
Then, if $\delta>0$ is taken small enough, for $z\in [-\delta,\delta]$
the hypersurfaces
$$
\Sigma_k(z)=\kappa_{k,z}(\wSi_k)=\{\tkappa(x',0;\xi',z),\ (x',\xi')\in \wSi_k\}
$$
are still transversal to the flow in $p^{-1}(z)$. Using this extension
we may continuously deform the
departure sets $D_{jk}$ into $D_{jk}(z)\defeq\kappa_{k,z}(\tD_{jk})\subset
\Sigma_k(z)$, and by consequence the tubes $T_{jk}$ into tubes
$T_{jk}(z)\subset p^{-1}(z)$ through a direct generalization of
\eqref{e:tubes}. The tube $T_{jk}(z)$ intersects $\Sigma_j(z)$ on the
arrival set $A_{jk}(z)\subset \Sigma_j(z)$; notice that for $z\neq 0$, the latter is in general
different from $\kappa_{j,z}(\tA_{jk})$
(equivalently $\tA_{jk}(z)=\kappa_{j,z}^{-1}(A_{jk}(z))$ is
generally different from $\tA_{jk}(0)$). These tubes induce a
Poincar\'e map $F_{jk,z}$ bijectively relating  $D_{jk}(z)$ with $A_{jk}(z)$.

The following Lemma, announced at the end of \S\ref{s:decompo}, 
shows that for $|z|$ small enough the interesting dynamics still takes
place inside these tubes: the 
trapped set is stable with respect to variations of the energy.
\begin{lem}\label{l:continu}
Provided $\delta>0$ is small enough, for any $z\in
  [-\delta,\delta]$ the trapped set $K_z\Subset \sqcup_{jk}T_{jk}(z)$. 

As a consequence, in this energy range the Poincar\'e map associated
with $\bSigma(z)$ fully describes the dynamics on $K_z$. 
\end{lem}
\begin{proof}
From our assumption in \S\ref{ass}, there exists a ball $B(0,R)$ (the
``interaction region'') such that, for any $E\in
[-1/2,1/2]$, the trapped set $K_E$ must be contained inside $T^*_{B(0,R)}X$.
If $R$ is large enough, any point $\rho\in p^{-1}(z)\setminus T^*_{B(0,R)}X$, $z\approx 0$,
will ``escape fast'' in the past or in the future, because the Hamilton vector
field is close to the one corresponding to free motion,
$2\sum_j\xi_j\partial_{x_j}$. Hence we only need to study the
behaviour of points in $p^{-1}(z)\cap T^*_{B(0,R)}\RR^n$. 

Let us define the
escape time from the interaction region $T^*_{B(0,R)}X$: for any $\rho\in T^*_{B(0,R)}X$,
$$
t_{esc}(\rho)\defeq \inf\{t>0,\ \max(|\pi_x\Phi^t(\rho)|, |\pi_x\Phi^{-t}(\rho)|) \geq R\}\,,
$$
For any $E\in[-1/2,1/2]$, the trapped set $K_E$ can be defined as the set of points in
$p^{-1}(E)$ for which $t_{esc}(\rho)=\infty$. 
Let us consider the neighbourhood of $K_0$ formed by the interior of
the union of tubes,
$(\sqcup T_{ik})^\circ$. By compactness, the escape time
is bounded from above outside this neighbourhood, that is in $p^{-1}(0)\cap
T^*_{B(0,R)}X\setminus(\sqcup T_{ik})^\circ$, by some finite $t_{1}>0$. 
By continuity of the flow $\Phi^t$, for $\delta>0$ small enough, the
escape time in the deformed neighbourhood $p^{-1}(z)\cap
T^*_{B(0,R)}X\setminus(\sqcup T_{ik}(z))^\circ$ will still be bounded
from above
by $2t_{1}$: this proves that $K_z\Subset \sqcup T_{ik}(z)$.
\end{proof}
A direct consequence is that the reduced trapped sets
$\TT_j(z)\defeq\Sigma(z)\cap K_{z}$ are contained inside $D_j(z)$.

For any set $S(z)$ depending on the energy in the interval $z\in [-\delta,\delta]$, we use the notation
\be\label{eq:energy-thick}
\widehat S \defeq \bigcup_{|z|\leq\delta} S(z)\,.
\ee
We will extend the notation to complex values of the parameter
$z\in \DOCh$, identifying $S(z)$ with $S(\Re z)$.


\subsubsection{Microlocal solutions near $\Sigma$}\label{s:micro-Sigma}

Let us now restrict ourselves to the neighbourhood of $\Sigma_k$, and drop the index $k$.
The canonical transformation $\tkappa$ can be locally quantized using the
procedure reviewed in \S\ref{fio},  
resulting in a microlocally defined unitary Fourier
integral operator 
\begin{equation}
\label{eq:U} U\; :\; H(\tOmega) \longrightarrow H(\Omega)\,, \ \ 
U^{*}\, P\, U\equiv hD_{x_n}  \,, \ \text{microlocally in $ \tOmega $.} 
\end{equation}
For $z\in \DOCh$, we consider the
microlocal Poisson operator
\begin{equation}\label{ml.1}
\mathrm{K}(z): L^2 ({\RR}^{n-1})\to L^2_{\rm{loc}} ({\RR}^n)\,,\quad [\mathrm{K}(z)\,v_+](x',x_n)=e^{ix_nz/h}\,v_+(x')\,,
\end{equation}
which obviously satisfies the equation $(hD_{x_n}-z)\,\mathrm{K}(z)\,v_+=0$.


For $ v_+ $ microlocally concentrated in a compact set, 
the wavefront set of $\mathrm{K}(z)\,v_+$ is not localized 
in the flow direction. On the other hand, the Fourier integral
operator $U$ is well-defined and unitary only from  $\tOmega$ to $\Omega$. Therefore, we use 
a smooth cutoff function $\chi_{\Omega}$, $\chi_{\Omega}=1$ in $\Omega$, $\chi_{\Omega}=0$ 
outside $\Omega'$ a small open neighbourhood of $\Omega$ (say, such that
$|x_n|\leq 2\eps$ inside $\tOmega'$), and 
define the Poisson operator
$$
K(z)\defeq\chi^w_{\Omega}\,U\,\mathrm{K}(z):H(\wSi)\to H(\Omega')\,.
$$
This operator maps any state $v_+\in H(\wSi)\subset L^2(\RR^{n-1})$, 
to a microlocal solution of the equation $(P-z)u=0$ in $\Omega$, with $u\in H(\Omega')$. 
As we will see below, the converse holds: each
microlocal solution in $\Omega$ is parametrized by a function $v_+\in H(\wSi)$. 

In a sense, the solution $u=K(z)v_+$ is an extension along the flow of the 
{\em transverse data} $v_+$. More precisely, $K(z)$ is a microlocally defined 
Fourier integral operator associated with the graph
\begin{equation}\label{eq:C_-}
C_-=\{(\tkappa(x',x_n,\xi',\Re z);x',\xi'),\ \ (x',\xi')\in\wSi,\ |x_n|\leq\eps\}\subset T^*(X \times \RR^{n-1})\,.
\end{equation}
Equivalently, this relation associates to each point $(x',\xi')\in \wSi$  a short trajectory segment through
the point $\tkappa(x',0;\xi',\Re z)\in \Sigma(\Re z)$.
We use the notation $C_-$ since this relation
is associated with the operator $R_-$ defined in \eqref{ml.6} below.

Back to the normal form $hD_{x_n}$, let us consider a smoothed out step function,
$$
\chi_0 \in \CI(\RR_{x_n}),\ \ \chi_0(x_n) =0\text{ for }x_n\le - \eps/2,\ \  
\chi_0 (x_n)=1\text{ for }x_n\ge  \eps/2\,.
$$ 
We notice that the
commutator $(i/h) [hD_{x_n},\chi_0]=\chi'_0 ( x_n)$ is localized in the 
region of the step and integrates to $1$: this implies the normalization property 
\begin{equation}\label{ml.2}
\la (i/h) [hD_{x_n},\chi_0 ]\mathrm{K}(z)v_+,\mathrm{K}(\bar{z})v_+ \ra
= \Vert v_+\Vert_{L^2 ( \RR^{n-1} ) } ^2 \,,
\end{equation}
where $ \langle \bullet , \bullet\rangle$ 
is the usual Hermitian inner product 
on $L^2(\RR^n)$. Notice that the right hand side is independent of the precise choice of $\chi_0$.

We now bring this expression to the neighbourhood of $\Sigma$ through the Fourier integral operator $\chi_{\Omega}^w U$.
This implies that the Poisson operator 
$K(z)$ satisfies:
\begin{equation}\label{ml.3}
\la (i/h)[P,\chi^w ]K(z)v_+,K(\bar{z})v_+\ra \equiv \Vert v_+\Vert^2\,\quad\text{ for any }v_+\in H(\wSi)\,.
\end{equation}
Here the symbol $\chi$ is such that $\chi^w\equiv U\,\chi_0^w\,U^*$ inside $\Omega$, so $\chi$ 
is equal to $0$ before $\Phi^{-\eps}(\Sigma )$ and equal to $1$ after $\Phi^{\eps}(\Sigma )$ (in the
following we will often use this time-like terminology referring to the flow $\Phi^{t}$).
In \eqref{ml.3}, we are only concerned with $[P,\chi^w ]$ microlocally near $\Omega$, since the
operator
$\chi_{\Omega}^wU$ is microlocalized in $\Omega'\times\tOmega'$. Hence, at this stage we can ignore
the properties of the symbol $\chi$ outside $\Omega'$. 

The expression \eqref{ml.3} can be written
\begin{equation}
\label{eq:KK} K ({\bar z })^* \,  [(i/h) P, \chi^w ] K(z) = Id \;:\; H(\wSi)\to H(\wSi)\,.
\end{equation}
Fixing a function $\chi $ with properties described after \eqref{ml.3} and
writing $ \chi =\chi _f$ (where $f$ is for {\em forward}), we define the operator
\begin{equation}\label{ml.4}
R_+ (z) \defeq K(\bar{z})^* \, [(i/h)P,\chi_f ] =\mathrm{K}(\bar z)^*\,U^*\chi_{\Omega}^w\, [(i/h)P,\chi_f ] 
\end{equation}
(from here on we denote $\chi=\chi^w$ in similar expressions).
This operator ``projects'' any $u\in H(\Omega)$ to a certain transversal function $v_+\in H(\wSi)$. 
But  it is important to notice that $R_+(z)$ is also well-defined on
states $u$ microlocalized in a small neighbourhood of the full trapped
set $\whK$:
the operator $\chi_{\Omega}^w\, [(i/h)P,\chi_f ] $ cuts off  the components of $u$ outside $\Omega$. Hence,
we may write
$$
R_+(z):H(\neigh(\whK))\to H(\wSi)\,.
$$

The equation 
\eqref{eq:KK} shows that this projection is compatible with the above
extension of the transversal function:
\begin{equation}\label{ml.5}
R_+ (z)\,K(z)=Id \; : \; H(\wSi)\to H(\wSi) \,.
\end{equation}
This shows that transversal functions $v_+\in H(\wSi)$ and microlocal solutions to $(P-z)u=0$ are bijectively related.
Since $|\Im z|\leq M_0 h \log (1/h)$ and $|x_n|\leq 2\eps$ inside
$\tOmega'$ (resp. $|x_n|\leq\eps$ inside $\tOmega$), we have the bounds
\[ 
\|K(z)\|_{L^2\to L^2} = \cO( h^{-2\eps M_0} )\,,\quad \|R_+(z)\|_{L^2\to L^2} = \cO( h^{-\eps M_0} ) \,.\quad 
\]
Just as $K(\bar z)^*$, $R_+(z)$ is a microlocally defined Fourier integral operator associated with the relation
\be\label{eq:C_+}
C_+=\{x',\xi';(\tkappa(x',x_n,\xi',\Re z)),\ \ (x',x_n,\xi',\Re z)\in\tOmega\}\subset T^*(\RR^{n-1} \times X)\,,
\ee
namely the inverse of $C_-$ given in \eqref{eq:C_-}. 
In words, this relation consists of taking any $\rho\in \Omega\cap p^{-1}(\Re z)$ and
projecting it along the flow on the section $\Sigma(z)$.

We now select a second cutoff function $\chi_b$ with properties similar with $\chi_f$, 
and satisfying also the nesting property
\begin{equation}
\label{so.5.0} 
\chi_b=1 \ \ \text{in a neighbourhood of }\supp\chi_f\,.
\end{equation}
With this new cutoff, we define the operator
\begin{equation}\label{ml.6}
R_-(z)u_-=[(i/h)P,\chi _b]\,K(z) \; : \; H(\wSi)\to H(\Omega)\,.
\end{equation}
Starting from a transversal data $u_-\in H(\wSi)$, this operator creates a microlocal solution in $\Omega$
and truncates by applying a pseudodifferential operator with symbol $H_p\chi_b$. Like $K(z)$, it is a microlocally 
defined 
Fourier integral operator associated with the graph $C_-$. its norm is
bounded by $\| R_-(z)\|_{L^2\to L^2}=\cO(h^{-\eps M_0})$.

\subsubsection{Solving a Grushin problem}\label{s:sol-Gru}
We are now equipped to define our microlocal Grushin problem in $\Omega$.
Given $v\in H(\Omega)$, $v_+\in H(\wSi)$, we want to solve the system
\begin{equation}\label{ml.7}
\begin{cases} ( i /h ) (P-z)u+R_-(z)u_-&=v, \\ 
R_+ (z)u&=v_+ \,,
\end{cases}
\end{equation}
with $u\in L^2(X)$ a forward solution, and  $u_-\in H(\wSi)$.

Let us show how to solve this problem.
First let $\wtu$ be the forward solution of
$( i /h ) (P-z) \wtu=v$, microlocally in $\Omega$.
That solution can be obtained using the Fourier integral
operator $ U $ in \eqref{eq:U} and the easy solution for $hD_{x_n} $.
We can also proceed using the propagator to define a forward parametrix:
\begin{equation}\label{ml.9}
\wtu \defeq E(z)\,v,\qquad E(z)\defeq\int_0^T  e^{-it(P-z)/h}\,dt.
\end{equation}
The time $ T $ is such that 
$\Phi^{ T } ( \Omega) \cap \Omega= \emptyset$ (from the above assumption on the separation between the $\Omega_k$ 
we may take $T=5\eps$).
By using the model operator $hD_{x_n}$, one checks that the parametrix $E(z)$ transports the wavefront set of $v$ as follows:
\be\label{eq:transport}
\WFh(E(z)v)\subset\WFh(v)\cup \Phi^{T}(\WFh(v))\cup\bigcup_{0\leq t\leq T}\Phi^{t}(\WFh(v)\cap p^{-1}(\Re z))\,.
\ee
In general, $\wtu$ does not satisfy $R_+(z)\wtu =v_+$, so we need to correct it.
For this aim, we solve the system
\begin{equation}\label{ml.9'}
\begin{cases}( i / h ) (P-z)\widehat{u} + R_-(z)u_- &\equiv 0\,,\\
R_+ (z)\widehat{u} &\equiv v_+ - R_+ (z)\wtu 
\end{cases}
\end{equation}
through the Ansatz
\begin{equation}\label{ml.10}
\begin{cases}u_-&= - v_+ + R_+ (z) \wtu\,,\\
\widehat{u}&=-\chi _b\, K(z)\,u_-\,.\\
\end{cases}
\end{equation}
Indeed, the property $ ( P -z )\, K(z) \equiv 0 $ ensures that $ ( i /h ) ( P - z ) \widehat u = - R_-(z) u_- $.
We then obtain the identities
\[ \begin{split} 
R_+ (z) \widehat u & = - K(\bar{z})^* \,  [(i/h)P,\chi_f ] \,
 \chi_b \,K(z)\, u_- \\
& \equiv - K(\bar{z})^* \,  [(i/h)P,\chi_f ]\, K(z)\, u_-
\\
& \equiv -u_-
\,.
\end{split}
\]
The second identity uses the nesting assumption $(H_p\chi_f)\chi_b=H_p\chi_f$, and the last one 
results from \eqref{eq:KK}. This shows
that the Ansatz \eqref{ml.10} solves the system  \eqref{ml.9'}.
Finally, $(u=\wtu+\widehat{u},u_-)$ solves \eqref{ml.7} microlocally in $\Omega\times\tSigma$, for
$ v \in H(\Omega)$ and $ v_+ \in H(\wSi)$ respectively.
Furthermore, these solutions satisfy the norm estimate
\begin{equation}\label{ml.11}
\Vert u\Vert+ \Vert u_-\Vert \lesssim h^{-5 M_0 \eps}(\Vert v\Vert + \Vert v_+\Vert) \,.
\end{equation}
The form of the microlocal construction in this 
section is an important preparation for the construction of our Grushin problem in the next section.
In itself, it only states that, for $ v $ microlocalized near $\Sigma$, $( i /h ) ( P - z ) u = v $ can be
solved microlocally near $ \Sigma $ in the forward
direction. 


\subsection{Microlocal solution near $ \whK$.}
\label{mnK}

We will now extend the construction of the Grushin problem near each $ \Sigma_k $, described in \S\ref{ml},
to obtain a microlocal Grushin problem near the full
trapped set $ \whK$. This will be achieved by relating the construction near $\Sigma_k$
to the one near the successor sections $\Sigma_j$.  We 
now need to restore all indices $k\in \{1,\ldots,N\}$ in our notations.

\subsubsection{Setting up the Grushin problem}
We recall that $ H (\wSi_k ) \subset L^2 ( \RR^{n-1} ) $ is the
space of functions microlocally concentrated in $ \wSi_k $ 
(see \eqref{eq:HV}).
For $u\in L^2(X)$ microlocally concentrated in $\neigh(\whK,T^*X)$,
we define
\ekv{so.1}
{
R_+(z)u=(R_+^1(z)u,...,R_+^N(z)u)\in H(\wSi_1)\times
...\times H(\tSigma_N) \,,
}
where each $ R_+^k ( z ):H(\neigh(\whK))\to H(\wSi_k) $ was defined in \S\ref{ml} using 
a cutoff $ \chi_f^k\in \CIc(T^*X) $ realizing a smoothed-out step from $0$ to $1$ along the flow
near $\Sigma_k$.

Similarly, we define
\be\label{so.2}
\begin{split}
R_-(z)\,:\,& H(\wSi_1)\times\ldots\times H(\wSi_N)\to H(\cup_{k=1}^N\Omega_k),\\
R_-(z)u_-&=\sum_1^N R_-^j(z)u_-^j,\qquad u_-=(u_-^1,...,u_-^N).
\end{split}\ee
Each $R_-^k(z)$ was defined in \eqref{ml.6} in terms of a cutoff function $\chi_b^k\in\CIc(T^*X)$ which also changes
from $0$ to $1$ along the flow near $\Sigma_k$, and does so before $\chi_f^k$. Below we will impose
more restrictions on the cutoffs $\chi_b^k$.

With these choices, we now consider the microlocal Grushin problem 
\ekv{so.4}
{
\begin{cases} (i/h) (P-z)u+R_-(z)u_- &\equiv v\,,\\ R_+(z)u &\equiv v_+ \,.\end{cases} 
}
The aim of this section is to construct a solution $(u,u_-)$
microlocally concentrated in a small neighbourhood of 
\[ 
K_0\times\kappa _1^{-1} (\TT_1)\times ...\times \kappa _N^{-1} (\TT_N) \,, 
\]
provided $(v,v_+)$ is concentrated in a
sufficiently small neighbourhood of the same set. 

To achieve this aim we need to put more constraints on the cutoffs $\chi_b^k$. We assume that each  
$ \chi_b^k\in \CIc (T^*X)$ is supported near the direct outflow of $\TT_k$. To give a precise condition,
let us slightly modify the energy-thick tubes $\whT_{jk}$ (see \eqref{e:tubes}, \eqref{eq:energy-thick}) 
by removing or adding some parts near their ends:
\[
\whT^{s_1 s_2}_{jk}\defeq \{\Phi^{t} (\rho)\,:\,\rho\in \whD_{jk},\ -s_2\,2\eps< t< t_+(\rho)+s_1\,2\eps\}\,,\qquad
s_i=\pm \,.
\]
With this definition, the short tubes $\whT^{--}_{jk}$ do not intersect the neighbourhoods $\Omega_k$, $\Omega_j$,
while the long tubes $\whT^{++}_{jk}$ intersect both (see Fig.~\ref{f:cutoffs}).

We then assume that 
\ekv{so.4.1}
{
\chi^k_b(\rho) = 1 \ \ \ \text{for} \ \ \rho\in \bigcup_{j\in J_+(k)} \whT^{--}_{jk}\,,
}
and $\supp\chi^k_b$ is contained in a small neighbourhood of that set.
Furthermore, we want the cutoffs $\{\chi_b^k\}_{k=1,\ldots,N}$
to form a {\em microlocal partition of unity} near $K_0$: there exists a 
neighbourhood $V_0$ of $\whK$ containing all long tubes:
\be\label{eq:V_0}
V_0\supset\bigcup_{k,j}\whT_{jk}^{++}\,,
\ee
and such that 
\ekv{so.5}
{
\sum_{k=1}^N \chi _b^k ( \rho ) \equiv 1   \ \ \ \text{for} \ \ \rho \in V_0 \,.
}

These conditions on $\chi_b^k$ can be fulfilled thanks to the assumption \eqref{eq:H4} on the section $\bSigma$. 
A schematic representation of these sets and cutoffs is shown in Fig.~\ref{f:cutoffs}.
\begin{figure}
\begin{center}
\includegraphics[width=0.9\textwidth]{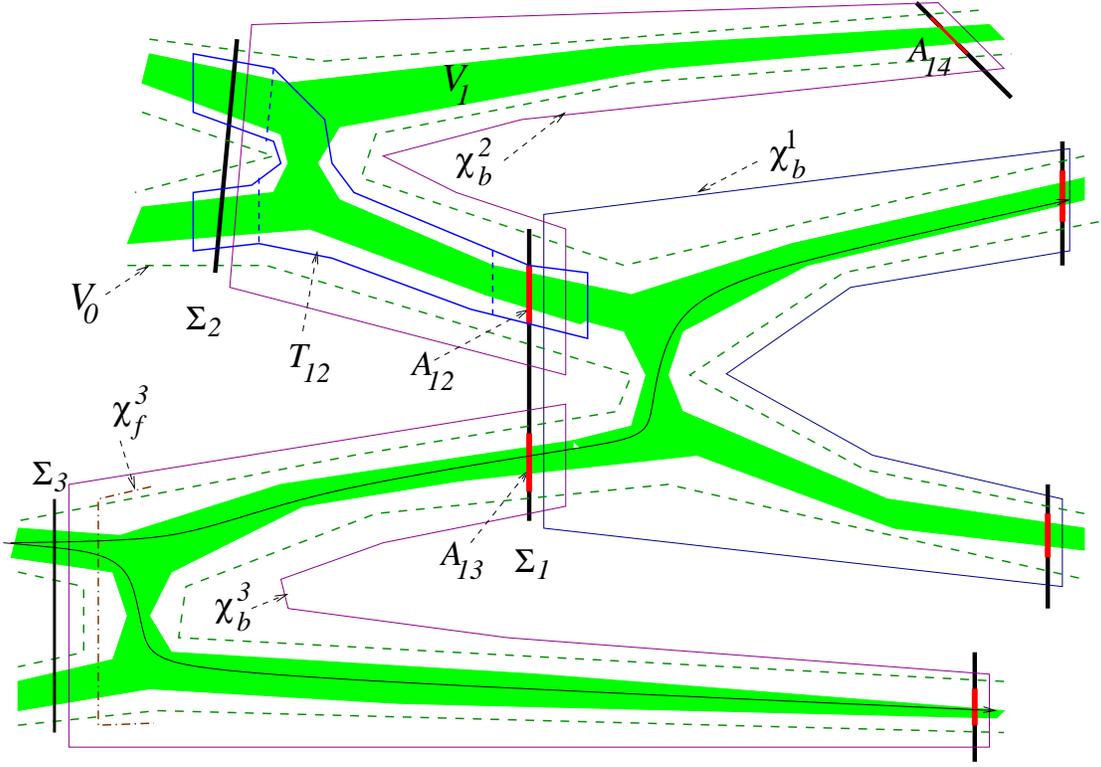}
\caption{\label{f:cutoffs} Schematic representation of (part of) the neighbourhoods  $V_1\subset V_0$
of $K_0$ (resp. green shade and green dashed contour),
some sections $\Sigma_k$ (thick black) and arrival sets $A_{kj}\subset\Sigma_k$ (red). We also show
the tubes $T^{\pm\pm}_{12}$ connecting
$\Sigma_2$ with $A_{12}$ (the dashed lines indicate the boundaries of $T_{12}^{--}$), the 
supports of the cutoffs $\chi_b^k$ and $\chi_f^{3}$ (dot-dashed line), and two trajectories in $K_0$ (full lines inside $V_1$). }
\end{center}
\end{figure}

\subsubsection{Solving the homogeneous Grushin problem}\label{s:homog-K0}

Let us first solve \eqref{so.4} when $v\equiv 0$. 
The wavefront set $\WFh(v_+^k)\subset\tSigma_k$ is mapped 
through $\kappa_{k,z}$ to a subset of $\Sigma_k(z)$.
The microlocal solution $K_k(z)v_+^k$, initially concentrated inside the neighbourhood $\Omega_k'$, 
can be extended along the flow to a larger set $\Omega_k^+$, which
intersects the successors $\Sigma_j(z)$ of $\Sigma_k(z)$ 
and contains the union of  tubes $\bigcup_{j\in J_+(k)}\whT^{++}_{jk}$ (we remind
that $j\neq k$ according to assumption \eqref{eq:H4}). This can be done
by extending the symplectomorphism $\tkappa_k$, the associated unitary Fourier integral operator $U_k$, and replace
the cutoff function $\chi_{\Omega_k}$ by a function $\chi_{\Omega_k^+}$ supported in the set $\Omega_k^+$; we can then
define the extended Poisson operator as:
$$
K^+_k(z)=\chi_{\Omega_k^+}^w\,U_k\,\mathrm{K}(z): H(\wSi)\to H(\Omega_k^+)\,.
$$
Assuming  $\kappa_{k,z}(\WFh(v_+^k))$
is contained in the departure set $D_k(z)\subset \Sigma_k(z)$,
the extended microlocal solution $K^+_k(z)v_+^k$ is concentrated in the union of tubes $\cup_{j\in J_+(k)}T_{jk}^{++}(z)$.
In that case, we take as our Ansatz
\begin{equation}
\label{eq:uk} 
u_k\defeq \chi _b^k\,K^+_k(z)\,v_+^k\,. 
\end{equation}
Set 
\be\label{e:t_max}
t_{\max}\defeq \max\{t_+(\rho),\ \rho\in \sqcup_k D_k(z),\ |\Re z|\leq \delta\}
\ee
the maximal return time for our Poincar\'e map. Then the above Ansatz
satisfies the estimate 
\be\label{e:uk-norm}
\| u_k \|_{L^2}\lesssim h^{-M_0(t_{\max}+\eps)}\,\|v_+\|_{H(\tD_k)}\,.
\ee

Due to the assumption \eqref{so.4.1}, the cutoff $\chi_b^k$ effectively truncates the solution only
near the sections $\Sigma_k(z)$ and $\Sigma_j(z)$, $j\in J_+(k)$, but not on the ``sides'' of 
$\supp\chi^k_b$. Hence, the expression
\ekv{eq:sol1}
{
(i/ h) (P-z)u_k\equiv  [(i/ h)P,\chi _b^k]\,K^+_k(z)\,v_+^k
}
can be decomposed into one component $R_-^k (z) v_+^k$ supported near $D_k(z)$,
and other components supported near the arrival
sets $A_{jk}(z)\subset\Omega_j$, due to the ``step down'' of $\chi^k_b$ near $A_{jk}(z)$. 
The assumption \eqref{so.5} ensures that
\ekv{eq:sol2}
{
[(i/ h)P,\chi _b^k]\equiv - [(i/ h)P,\chi _b^j]\quad\text{microlocally near } A_{jk}(z)\,,
}
so the expression in \eqref{eq:sol1} reads
\ekv{eq:sol3}
{
(i/ h) (P-z)u_k \equiv  R_-^k (z) v_+^k - \sum_{j\in J_+(k)} [(i/ h)P,\chi _b^j]\,K^+_k(z)\,v_+^k\,.
}
Now, for each $j\in J_+(k)$ we notice that $K^+_k(z)\,v_+^k$ is a solution of $(P-z)u=0$ near $A_{jk}(z)$,
so this solution can also be 
parametrized by some transversal data ``living'' on the section $\Sigma_j(z)$ (see the discussion before \eqref{eq:C_-}). 
This data obviously depends linearly on $v_+^k$, which defines the {\em monodromy operator} $\cM_{jk}(z)$:
\begin{equation}
\label{eq:defM}
 K^+_k(z) v_+^k\equiv K_j(z)\,\cM_{jk}( z ) v_+^k\,, 
\quad \text{microlocally near $A_{jk}(z)$\,.  } 
\end{equation}
The operators $ \cM_{jk} (z) $ are microlocally defined from $\tD_{k}\subset\tSigma_k$ to 
$\tA_{jk}(z) \subset\tSigma_j$, they are zero on $H(\tD_{\ell k})$ for $\ell\neq j$. 
The identity \eqref{eq:KK} provides an explicit formula:
\begin{equation}
\label{eq:defM'}
\cM_{jk} ( z) = K_j(\bar z )^* \, [(i/h) P , \chi^j_f ] K^+_k ( z )= R_+^j ( z )K^+_k(z) \,.
\end{equation}
Before further describing these operators, let us complete the solution of our Grushin problem.
Combining \eqref{eq:sol3} with \eqref{eq:defM}, we obtain
\be\label{eq:sol4}
(i/ h) (P-z)u_k \equiv  R_-^k (z) v_+^k - \sum_{j\in J_+(k)}R_-^j ( z ) \cM_{jk}( z ) v_+^k\,.
\ee
This shows that the problem \eqref{so.4} in the case $v=0$ and a single $v_k^+$, 
$\WFh(v_+^k)\subset\tD_k$ is solved by
$$
u\equiv \chi _b^k\,K^+_k(z)\,v_+^k,\quad u_-^k=-v_+^k,\quad u_-^j=\cM_{jk} ( z) v_+^k,\ \ j\in J_+(k)\,.
$$
We now consider the Grushin problem with $v=0$, $v_+=(v_+^1,\ldots,v_+^N)$ 
with each $v_+^k$ microlocalized in $\tD_k$.
By linearity, this problem is solved by
\be\label{eq:sol-homog}
\begin{split}
u&\equiv \sum_k\chi _b^k\,K^+_k(z)\,v_+^k,\\
u_-^j&\equiv  - v_+^j  + \sum_{k\in J_-(j)}\cM_{jk} ( z) v_+^k \,.
\end{split}\ee
From the above discussion, $u$ is microlocalized in the neighbourhood $V_0$ of $\whK$, while $u_-^j$ is
microlocalized in $\tD_{j}\cup\tA_{j}(z)$.

Let us now come back to the monodromy operators. 
The expression \eqref{eq:defM'} shows that $\cM_{jk}(z)$ is a microlocal
Fourier integral operator. Since we have extended the solution $K_k(z)\,v_+^k$ beyond $\Omega_k$, 
the relation associated with the restriction of $K^+_k(z)$ on $H(\tD_{jk})$ is a modification of \eqref{eq:C_-},
of the form
$$
C_-^{jk}=\{(\Phi^{t}(\tkappa_{k,z}(\rho));\rho),\ \rho\in \tD_{jk},\ \ -\eps\leq t\leq t_{\max}+\eps\}\,,
$$
such that the trajectories cross $\Sigma_j$. On the other hand, the relation $C_+$
associated with $R_+^j(z)$ is identical with \eqref{eq:C_+}. 
By the composition rules, the relation associated with $\cM_{jk}(z)$ is 
$$
C^{jk}=\{(\rho',\rho),\ \ \rho\in \tD_{jk},\
\rho'=\kappa_{j,z}^{-1}\circ F_{jk,z}\circ \kappa_{k,z}(\rho)= \tF_{jk,z}(\rho)\}\,.
$$
This is exactly the graph of the 
Poincar\'e map $F_{jk,z}:D_{jk}(z)\to A_{jk}(z)$, seen through the coordinates charts $\kappa_{k,z}$, $\kappa_{j,z}$.

When $z$ is real, the identity \eqref{eq:KK} implies that $\cM_{jk}(z):H(\tD_{jk})\to H(\tA_{jk}(z))$ is microlocally
unitary. Also, the definition \eqref{eq:defM'} shows that this operator depends
holomorphically of $ z $ in the rectangle $\DOCh$. 
To lowest order, the $z$-dependence takes the form
$$
\cM_{jk}(z)=\cM_{jk}(0)\,\Op(\exp(iz\tilde{t}_{+}/h))\,\big(1+\cO(h\log(1/h))\big)
$$
where $\tilde{t}_+\in \CIc(\RR^{n-1};\RR_+)$ is an extension of the
return time associated with the map $\tF_{jk,z}$ on $\tD_{jk}$. 
For $z\in\DOCh$, this operator satisfies the asymptotic bound
\be\label{e:M-norm-bound}
\| \cM_{jk}( z) \|_{H(\tD_k)\to H(\tA_j(z))}=\cO(h^{-M_0 t_{\max}})\,.
\ee

\subsubsection{Solving the inhomogeneous Grushin problem}\label{s:inhomog-K0}

It remains to discuss the inhomogeneous problem
\ekv{so.7}
{
(i/h) (P-z)u+R_-u_- \equiv v\,,
}
for $v$ microlocalized in a neighbourhood $V_1$ of $\whK$, which satistfies
\be\label{eq:V_1}
V_1\subset \bigcup_{j,k}\whT^{-+}_{jk}\,.
\ee 
(each tube $\whT^{-+}_{jk}$ intersects $\Omega_k$ only near $\whD_k$, see figure~\ref{f:cutoffs}).

Let us first assume that $v$ is microlocally concentrated inside a short tube $\whT_{jk}^{--}$.
We use the forward parametrix  $E(z)$ of $(i/h)(P-z)$ given in
\eqref{ml.9} with the time
\be\label{eq:T}
T= t_{\max}+5\eps\,,
\ee
and consider the Ansatz
\be\label{eq:uu}
u\defeq \chi _b^k\,E(z)\,v\,.
\ee
According to the transport property \eqref{eq:transport},
$E(z)v$ is microlocalized in the outflow of $\whT_{jk}^{--}$, so the cutoff $\chi_b^k$
effectively truncates $E(z)v$ only near $A_{jk}(z)\subset \Omega_j$. 
The partition of unity \eqref{so.5} then implies that
$$
(i/ h) (P-z)u \equiv v+ [(i/h)P,\chi _b^k]\,E(z)\,v\, \equiv v - [(i/h) P,\chi _b^j]\,E(z)\,v\,.
$$
Also, the choice of the time $T$ ensures that $E(z)v$ is a microlocal solution of $(P-z)u=0$ near $A_{jk}(z)$,
so
$$
E(z)v\equiv K_j(z) R_+^j ( z )E(z)v \quad\text{microlocally near }A_{jk}(z)\,.
$$
Thus, we can solve (\ref{so.7}) by taking
$$
u_-^j\equiv R_+^j ( z ) E(z)v\,,\quad u_-^\ell=0,\ \  \ell\neq j\,.
$$
The propagation of 
wavefront sets given in \eqref{eq:transport} 
shows that $u_-^j\in H(\tA_{jk}(z))$, 
and that $\WFh(u)\subset \whT_{jk}^{+-}$ 
does not intersect the ``step up'' region of the forward cutoffs $\chi_f^\ell$, so that $R_+^\ell(z) u\equiv 0$
for all $\ell=1,\ldots,N$. 

If $v$ is
microlocally concentrated in $V_1\cap\cup_{|t|\leq\eps}\Phi^t(\whD_{k})$, we can replace
the cutoff $\chi _b^k$ in \eqref{eq:uu} by 
$$\chi _b^k+\sum_{\ell\in J_-(k)} \chi _b^\ell\,, $$
 and apply the same construction. The only notable difference
is the fact that $R_+^k(z)u$ may be a nontrivial state concentrated in $\cup_{|t|\leq\eps}\whD_{k}$.

In both cases, we see that 
$\Vert u\Vert+\Vert u_-\Vert \lesssim h^{-M_0 (t_{\max}+2\eps)}\Vert
v\Vert$. By linearity, the above procedure allows to solve \eqref{so.7} for any $v$ microlocalized inside the
neighbourhood $V_0$.

This solution produces a term $R_+u$, which can be solved
away using the procedure of \S\ref{s:homog-K0}. Notice that
$\Vert R_+u\Vert\lesssim h^{-M_0 (t_{\max}+\eps)}\Vert v\Vert$.

\medskip

We summarize the construction of our microlocal Grushin problem in the following
\begin{prop}\label{so1}
For $\delta>0$ small enough, there exist neighbourhoods of $\whK=\whK_\delta$ in $T^*X$,  $V_+$ and $V_-$,
and neighbourhoods of $ \tkappa_{j}^{-1}(\widehat\TT_j)$ in $\tSigma_j$, 
$V_+^j $, and $ V^j_- $, $ j = 1, \cdots \, N $, such that for any  
\[ (v,v_+)\in H ( V_+ ) \times H ( V^1_+ ) \times \cdots H ( V^N_+ ) \,, \]
we can find
\[ (u,u_-) \in H ( V_- ) \times H ( V^1_-) \times \cdots H ( V^N_- ) \,, \]
satisfying
\[
\frac i h (P-z)u+R_-(z)u_- \equiv v\,,\qquad R_+(z)u \equiv v_+ \, \quad\text{microlocally in }V_+\times V^1_+\times\cdots V^+_N\,.
\]
Here $ R_\pm ( z ) $ are given by 
\eqref{so.1} and \eqref{so.2}. Furthermore, the solutions satisfy the norm estimates
\[
\Vert u\Vert+\Vert u_-\Vert \lesssim h^{-M_0 (2t_{\max} + 2\eps)}
(\Vert v\Vert + \Vert v_+ \Vert) \,,
\]
where $t_{\max}$ is the maximal return time defined in \eqref{e:t_max}. 

One possible choice for the above sets is
$$
V_+=V_1,\ \ V_-\defeq V_0,\ \ V_+^k=\tD_k,\ \ V_-^k=\tD_{k}\cup\bigcup_{|\Re z|\leq \delta}\tA_{k}(z)\,.
$$
\end{prop}
\begin{proof}
Take 
$v\in H(V_1)$, and call $(\wtu,\wtu_-)$ 
the solution for the inhomogeneous problem \eqref{so.7}.  
Then the propagation estimate \eqref{eq:transport}
implies that $\wtu$ is concentrated inside the 
larger neighbourhood 
$V_0\subset\cup_{j,k}\whT_{jk}^{++}$ (see \eqref{eq:V_0}), while $\wtu_-^j\in H(\tA_j(z))$.

We have $R_+^k(z)\wtu\in H(\tD_k)$ so, provided the data satisfies
$v_+^k\in \tD_k$, 
the computations of \S\ref{s:homog-K0} show how to solve the homogeneous
problem with data $(v_+-R_+(z)\wtu)$. That solves the full problem. 
The expressions \eqref{eq:sol-homog} show that 
the solutions to the homogeneous problem $(\widehat{u},\widehat{u}_-^k)$ are microlocalized, 
respectively, in 
$V_0$ and in $\tD_{k}\cup\tA_{k}(z)$.
\end{proof}

\medskip

\noindent
{\bf Remark.}
The proof of the proposition shows that the neighbourhoods $ V^k_+$ and $V^k_-$ 
are different. 
For given data $(v,v^+)$, the solutions
$(u,u_-)$ will not in general be
concentrated in the same small set as the initial data. This, of course, reflects
the fact that a neighbourhood $V$ of $K_0$ is not invariant under the forward flow, but
escapes along the unstable direction.
In order to transform the microlocal Grushin problem described in this proposition into a
well-posed problem, we need to take care of this escape phenomenon. This will be done
using escape functions in order to deform the norms on the spaces $L^2(X)$ (as described in \S\ref{cef}),
but also on the auxiliary spaces $L^2(\RR^{n-1})$.



\section{A well posed Grushin problem}
\label{wpg}

The difficulty described in the remark at the end of \S \ref{gps} 
will be resolved by modifying the norms on the space $L^2(X)\times L^2(\RR^{n-1})^N$,
through the use of exponential 
weight functions as described in \S \ref{cef}. These weight functions will
be based on the construction described in \S\ref{eft}.

In most of this section we will consider the scaled operator $P_{\theta,R}$ globally,
so we cannot replace it by $P$ any longer. To alleviate notation, we will write this operator
\begin{equation}
\label{eq:scaleR}   P = P_{\theta , R } \,, \quad\theta = M_1 h \log ( 1/h) \,, \quad
R \gg C_0 \,, 
\end{equation}
where $C_0 $ is the constant appearing in \eqref{ge.2}, and $M_1>0$ is
a constant (it will be required to 
satisfy \eqref{e:theta} once we fix the weight $G$, and is larger than
$M_0$ appearing in Theorem~\ref{t:s}). 

We will first discuss the local construction near each $ \Sigma_k$ 
and then, as in the previous section, adapt it to construct
a global Grushin problem.

Our first task is still microlocal: we explain how a deformation of
the norm on $L^2(X)$ by a suitable weight function $G$ can be
used to deform the norms on the $N$ auxiliary spaces $L^2(\RR^{n-1})$,
microlocally near $\tSigma_k$.

\subsection{Exponential weights near $ \Sigma_k $.}
\label{ewn}

As in \S\ref{ml}, in this subsection 
we work microlocally in the neighbourhood $\Omega_k$ of one component $\Sigma_k$ ($\Omega_k$ 
is the neighbourhood described in \S\ref{ml}); we 
drop the index $k$ in our notations.
Notice that the complex scaling has no effect in this region, so $P\equiv P_{\theta,R}$. We will impose a constraint
on the weight function $G$ near $\Sigma$, and construct a weight functions $g$ on $\tSigma$. The
construction of the local solution performed in \S\ref{ml} will then be studied in these deformed spaces.

Take a function $g^0\in \CIc(\IR^{n-1})$, and use it to define $\wtG_0\in\CI(T^*\RR^n)$, so that
$$
\wtG_0(x',x_n,\xi',\xi_n) = g^0(x',\xi')\quad\text{in }\tOmega'.
$$
Then, using the Fourier integral operator $U$ given in \eqref{eq:U}, 
one can construct a weight function $G_0\in S(T^*X)$ such that
$$
G_0^w \equiv U\,(\wtG_0)^w\,U^*\quad\text{microlocally near }\Omega.
$$
Notice that $G_0$ now depends on $h$ through an asymptotic expansion
\ekv{ad.0.2}{
G_0(h)\sim \sum_{j\geq 0} h^j\,G_{0,j}\,,\quad G_{0,j}\in \CIc(T^* X)\text{ independent of }h\,.
}
This weight satisfies $G_{0,0}=\wtG_0\circ\tkappa^{-1}$ in $\Omega$, and the invariance property
\begin{equation}\label{ad.1}
[P(h),G_0^w( x, h D ) ] \equiv 0\ \quad\text{microlocally in }\Omega\,.
\end{equation}
As in \S\ref{cef}, we rescale these weight functions by
\ekv{eq:G_z}
{G\defeq M h \log (1/h)\,G_0,\qquad g\defeq M h \log(1/h)\, g^0\,.
}
Still using the model $hD_{x_n}$, one can easily check the intertwining property
\begin{equation}
\label{eq:GK}
\begin{split}
G^w(x,hD_x;h)\,K(z)&\equiv K(z)\,g^w(x',hD_{x'};h):H(\tSigma)\to H(\Omega')\,,\\
e^{-G^w(x,hD_x;h)/h}\,K(z)&\equiv K(z)\,e^{-g^w(x',hD_{x'};h)/h}:H(\tSigma)\to H(\Omega')\,.
\end{split}
\end{equation}
Using the weights $G$ and $g$ we define the microlocal Hilbert spaces
$H_G(\Omega')$ and $H_{g}(\tSigma)$ by the method of 
\S\ref{cef}. We need to check that the construction of a microlocal solution performed in \S\ref{s:micro-Sigma} 
remains under control with respect to these new norms.
\begin{lem}\label{ad1}
The operators
\[ 
K(z) \; : \; H_{g}( \wSi) \to H_G ( \Omega' ) \,, \quad z \in \DOCh
\]
satisfy the analogue of \eqref{ml.3}. Namely, taking a cutoff $\chi$ jumping from $0$ to $1$ near $\Sigma$ 
as in \S\ref{s:micro-Sigma}, then any $ v_+ \in H_g(\wSi)$ will satisfy
\begin{equation}
\label{eq:pad1}
\langle [(i/h) P,\chi^w]\,K(z)\, v_+ , K(\bar{z})\,v_+ \rangle_{H_G} 
\equiv \| v \|^2_{H_g} \,.
\end{equation}
\end{lem}
\begin{proof}
From the cutoff $\chi$ we define the deformed symbol $\chi_G$ through
\[ 
\chi_G^w ( x, h D )  \defeq e^{-G^w ( x , h D ) / h } \, \chi^w ( x, h D ) \, e^{G^w ( x, h D) / h} \,.
\]
The symbol calculus of \S\ref{cef} shows that $ \chi_G $ also jumps from $0$ to $1$ near $\Sigma$, so that
(returning to the convention of using $\chi$ for $\chi^w$)
\begin{eqnarray*}
\langle[(i/h)P,\chi ]K(z)v_+,K(\bar{z})v_+ \rangle _{H_G}
& \equiv & \langle e^{-G/h}[(i/h)P,\chi ]K(z)v_+ , e^{-G/h}K(\bar{z}) v_+\rangle_{L^2} \\
& \equiv & \langle K(\bar{z})^* [(i/h)P_G,\chi_G]\,K(z)\, e^{-g/h}\,v_+,e^{-g/h}\,v_+\rangle_{L^2} \\
& \equiv & \langle K(\bar{z})^* [(i/h)P,\chi_G]\,K(z)\, e^{-g/h}\,v_+,e^{-g/h}\,v_+\rangle_{L^2} \\
& \equiv & \Vert e^{-g/h}\,v_+\Vert^2 \equiv \Vert v_+\Vert^2_{H_{g}} \,. 
\end{eqnarray*}
In the second line we used \eqref{eq:GK}, the third line results from $P\equiv P_G$, due to \eqref{ad.1},
and the last one from \eqref{ml.3} applied to $\chi_G$.
\end{proof}

Equation \eqref{eq:GK} shows that, for $z\in \DOCh$, the operators
$K(z)$ and $R_\pm(z)$ defined respectively
in \eqref{ml.4} and \eqref{ml.6}, satisfy the same norm estimates with respect to
the new norms as for the $L^2$ norms:
\begin{align}\label{e:norm_K}
\| K(z)\|_{H_{g}(\tSigma)\to H_G(\Omega)} &= \cO( h^{-M_0
  \eps})\,,\\
\label{e:R_+-G}
\|R_+( z )\|_{H_G(\Omega) \to H_g(\tSigma)}  = \cO( h^{-M_0 \eps})\,,&\qquad 
\|R_-( z)\|_{H_g(\tSigma) \to  H_G(\Omega)} = \cO( h^{-M_0 \eps})\,.
\end{align}
The arguments presented in \S\ref{ml} carry over to the weighted
spaces, and the microlocal solution to the problem \eqref{ml.7} constructed in
\S\ref{s:sol-Gru} satisfies the norm estimates
\be\label{ad.2}
\Vert u\Vert_{H_G}+\Vert u_-\Vert_{H_g} \lesssim h^{-5 M_0 \eps} \big(\Vert v\Vert _{H_G}+ \Vert v_+\Vert_{H_g}\big)\,.
\ee
Given a function $G_{0,0}(x,\xi)$ satisfying $H_pG_{0,0}=0$ in $\Omega$, one can iteratively construct 
a full symbol $G_0$ of the form \eqref{ad.0.2}, such that \eqref{ad.1} holds.
Now, the lower order terms in $G_0$ may change the norms only by
factors $\big(1+\cO(M h\log (1/h))\big)$,
so the same norm estimates hold if we
replace $G_0$ by its principal symbol $G_{0,0}$ in the definition of the new norms. As a result, we get the following 
\begin{prop}\label{p:G_0}
Take $g^0\in \CIc(T^*\IR^{n-1})$, $\wtG_0(x',x_n,\xi',\xi_n)= g^0(x',\xi')$, $G_0\in \CIc(X)$ satisfying $G_0=\wtG_0\circ\tkappa$ in $\Omega$, and
$$
G=Mh\log(1/h)\,G_0,\qquad g = Mh\log(1/h)\,\wtG_0\,.
$$
Then, the estimates (\ref{e:norm_K}--\ref{ad.2}) hold in the spaces $H_G(\Omega)$, $H_{g}(\tSigma)$.
\end{prop}

\subsection{Globally defined operators and finite rank weighted spaces}
\label{frR}

In this section we transform our microlocal Grushin problem into a globally defined one. This will require transforming all
the microlocally defined operators ($R_{\pm}(z)$, $\cM_{jk}(z)$) into globally defined operators acting on $L^2(X)$
or $L^2(\RR^{n-1})$. Because our analysis took place near the trapped set $K_0$, we will need to restrict our auxiliary operators
to some subspaces of $L^2(\RR^{n-1})$ obtained as images of some finite rank projectors. These subspaces are composed of
functions microlocalized near $K_0$. 
To show that the resulting Grushin problem is well-posed (invertible), the above construction must be performed using 
appropriately deformed norms on the spaces $L^2(X)$, $L^2(\RR^{n-1})$, 
obtained by using globally defined weight functions $G$, $g_{j}$. Our first task is thus to complete the constructions of
these global weights, building on \S\ref{eft} and \S\ref{ewn}.

\subsubsection{Global weight functions}
We will now construct global weight functions $G\in \CIc(X)$, $g_{j}\in \CIc(T^*\RR^{n-1})$ (one for each section $\Sigma_j$). 
For this, we will use the construction of an escape function away from $K_0$ presented in \S\ref{eft}, and modify it 
near the Poincar\'e section so that it takes the form required in Proposition \ref{p:G_0}, and allows us to
define auxiliary escape functions $g_{j}$. These
weight functions will allow us to to define {\em finite rank} realizations of the microlocally 
defined operators $R_\pm ( z )$ and $\cM(z)$.

Our escape function $G_0\in S(T^*X)$ is obtained through a slight modification of the weight 
$ G_2(x,\xi) $ described in \eqref{eq:gsa}. The modification only takes place near the trapped set $\whK$, and
in particular near the sections $\Sigma_j$. The following lemma is easy to verify.
\begin{lem}
\label{l:G0}
Let $\{\Omega_j,\}_{j=1,\ldots,K}$ be the neighbourhoods of $\Sigma_j$ described in \S\ref{s:normal},
$\Omega_j'$ and $\Omega''_j$ be small neighbourhoods of $\Omega_j$, $\Omega_j\Subset \Omega_j'\Subset \Omega_j''$,
and let $ V $ be a small neighbourhood of $ \whK_\delta $ (see \eqref{eq:Khd}).
Then there exists $ G_0 \in \CIc(T^*X )$ such that
\be
\label{eq:gs0} 
\begin{split}
& H_p G_0 \geq 1  \ \  \text{ on }\ \  
 T^*_{B(0,3R )} X \cap p^{-1}([-\delta,\delta]) \setminus W,\quad  W\defeq V \cup \bigcup_{j=1}^N \Omega''_j\,, 
\\ &H_p G_0 =  0 \ \ \text{ on\ \  $ \Omega_j' $,} 
\\
& H_p G_0 \geq 0  \quad \text{on \ \  } T^*_{B(0,3R )}  X ,  \\
& H_p G_0   \geq -\delta_0  \quad \text{ on $ T^*X $.} 
\end{split}
\ee
Besides, using the coordinate charts $\tkappa_j:\tOmega_j'\to\Omega_j'$ (see \S\ref{s:normal}), 
we can construct $G_0$ such that $G_0\circ \tkappa_j\rest\tOmega_j'$ is 
independent of the energy variable $\xi_n\in[-\delta,\delta]$.
\end{lem}
The last assumption (local independence on $\xi_n$) is not strictly
necessary, but it simplifies our construction below,
making the auxiliary functions $g_j$ independent of $z$ --- see Proposition \ref{p:G_0}.

For the set $V$ we assume that $V\Subset V_1$, where $V_1$ is the set
defined in \eqref{eq:V_1}. As a consequence, there exists a set
$V'_1$, with  $V\Subset V'_1\Subset V_1$ with the following property.
Consider the the parametrix
$E(z)$ \eqref{ml.9} with the time $T=t_{\max}+5\eps$. Then there exists $t_1>0$ such that, for any
$\rho\in p^{-1}([-\delta,\delta])\setminus V'_1$, the trajectory
segment $\{\Phi^t(\rho),\ 0\leq t\leq T\}$ spends a time
$t\geq t_1$ {\em outside of} $W$.
The main consequence of this property is a strict increase of the
weight along the flow outside $V'_1$:
\be\label{e:V-V_1}
\forall \rho\in T^*_{B(0,2R )}X\cap p^{-1}([-\delta,\delta])\setminus V'_1\,,\qquad G_0(\Phi^T(\rho))-G_0(\rho)\geq t_1\,.
\ee
(Here we use the fact that $T$ is small enough, so that a particle of energy $z\approx 0$ starting inside 
$ T^*_{B(0,2R )}$ at $t=0$
will remain inside $T^*_{B(0,3R )}$ up to $t=T$.)
The set $V$ will be further characterized in the next subsection.

From now on, we will take for weight function $G=Mh\log h\,G_0$ with such a function $G_0$, and use it
to define a global Hilbert norm $\|\bullet\|_{H^k_G(X)}$ as in \eqref{e:deformed-norms}.
As in Proposition \ref{p:G_0}, 
we define, for each $j=1,\ldots,N$, the auxiliary weight
\be\label{e:G^j}
g_j(x',\xi')\defeq M h \log (1/h)\, 
G_0\circ\tkappa_j (x',0,\xi ',0) \,, \quad ( x' , \xi') \in \tSigma_j\,, 
\ee
and extend it to an element of $\CIc ( T^* \RR^{(n-1)})$,
so that the deformed Hilbert norm 
$$
\| v \|_{H_{g_j}} = \| e^{-g_j^w(x',hD_x')/h}\, v\|_{L^2(\RR^{n-1})}
$$ 
is globally well-defined. 
Proposition \ref{p:G_0} shows that our microlocal construction near $\Sigma_j$ satisfies
nice norm estimates with respect to the spaces $H_G(X)$, $H_{g_j}$.


To see the advantages of having weights which are escape functions 
we state the following lemma which results from applying
Lemma \ref{l:TGV} to the Fourier integral operator  $ \exp ( - i t P / h ) $:
\begin{lem}\label{l:G}
Suppose that 
$ \rho_1 =\Phi^{t} (\rho _0) $ for some $t>0$, 
and that 
$$ \Delta \defeq G_0(\rho _1) - G_0(\rho _0) > 0\,. $$
Suppose also that $\chi _j\in \CIc (T^*X)$,
$j=0,1$, have their supports in small neighbourhoods of $\rho _j$'s. 
Then for $h$ small enough we have
\begin{equation}
\label{eq:lG}
\|e^{-itP/h}\,\chi^w_0\|_{{H_G} \rightarrow {H_G}} \le h^{M\Delta/2}\,, \qquad \|\chi^w_1\,e^{-itP/h}\|_{{H_G} \rightarrow {H_G}} \le h^{M\Delta/2}\,.
\end{equation} 
\end{lem}

\subsubsection{Finite dimensional projections}\label{s:finite-dim}
We want to construct a {\em finite dimensional} subspace 
of the Hilbert space $ H_{g_j } ( \RR^{n-1} ) $, such that the microlocal 
spaces $ H_{g_j} ( V_\pm^j ) $ are both
approximated by it modulo $ \cO ( h^\infty ) $. 

For each $j=1,\ldots,N$, let $S'_j, S_j$ be two families of open sets
with smooth boundaries in $T^*\RR^{n-1}$, satisfying
\be\label{e:S_j}
\tkappa_j^{-1}(\widehat{\TT}_j) \Subset S'_j \Subset S_j \subset \tD_j  \,,\quad j=1,\ldots,N\,.
\ee
In particular, each $S_j$, $S'_j$ splits into disjoint components $S_{kj}'\Subset S_{kj}\subset \tD_{kj}$.

Once these sets are chosen, we need to adjust the set $V$ in
Lemma~\ref{l:G0}, making it thinner if necessary:
\begin{lem}\label{l:t_0}
For $\delta>0$ small enough, there exists 
$V=\neigh(\whK_\delta,V_1)$ and $t_0>0$ such that the following property holds. 

For any indices $j=1,\ldots,N$, $k\in J_+(j)$, any $z\in [-\delta,\delta]$ and any point 
$\rho\in \tD_{kj}\cap S_j$ such that its successor 
$\tF_{kj,z}(\rho)$ does not belong to $S'_k$, then the trajectory
between $\kappa_{j,z}(\rho)$ and $F_{kj,z}(\kappa_{j,z}(\rho))$
spends a time $t\geq t_{0}$ outside of $W=V \cup \bigcup_{j=1}^N \Omega''_j$.
\end{lem}
The time $t_0$ is necessarily smaller than the maximal
return time $t_{\max}$ of \eqref{e:t_max}; on
the other hand, $t_0$ increases if we decrease the width $\sim\eps$ of
the sets $\Omega''_j$.
See figure~\ref{f:setV} for a sketch.
\begin{figure}\begin{center}
\includegraphics[angle=-00,width=.85\textwidth]{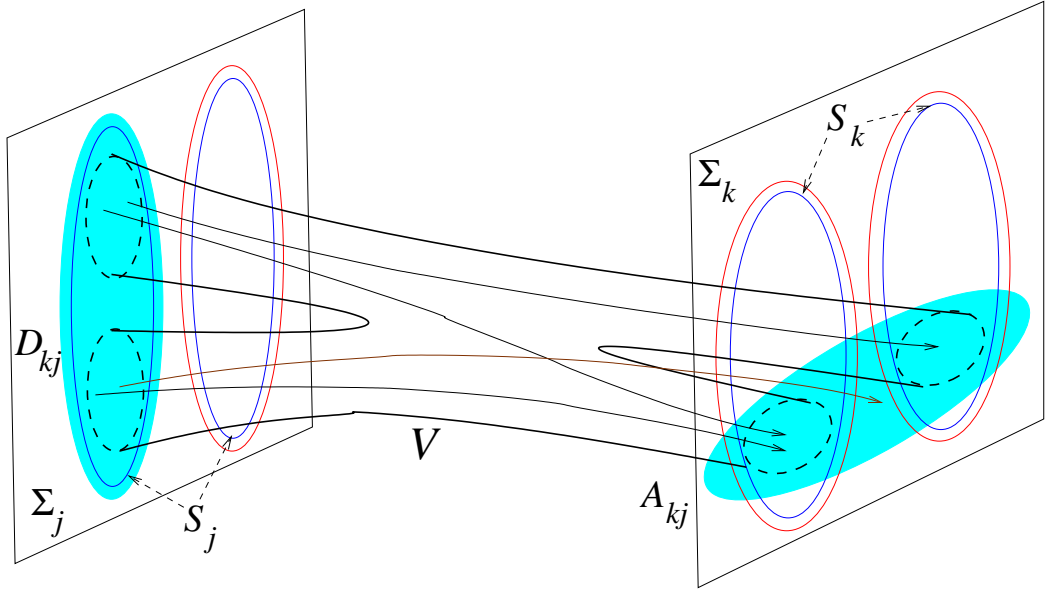}
\caption{\label{f:setV} Schematic representation (inside some energy layer $p^{-1}(z)$) 
of the neighbourhood $V$ and the sets $S_k$, $S_j$. 
The departure/arrival sets $D_{kj}$, $A_{kj}$ are similar to the ones
appearing in figure~\ref{fig:A-D}. The sets $S_{k}$, $S_{j}$ are
represented through their images in $\Sigma_k,\,\Sigma_j$ through
$\kappa_{k,z},\,\kappa_{j,z}$. We showe $3$ trajectories staying
inside $V$ all the time, and one ending outside $S_k$.}
\end{center}\end{figure}
Now, let
$$ 
Q_j=Q_j(x',\xi';h)\in S ( T^* \RR^{n-1})\,, 
$$ 
with leading symbol $q_j$ independent of $h$ (the leading symbol is meant in the
sense of \eqref{eq:leads}). We choose that leading
symbol to be real and have the following properties:
\begin{equation}
\label{gl.2.5}
\begin{split}
& q_j(\rho )<0\,,\ \ \ \rho \in S_j\,,\\
& q_j(\rho )>0\,,\ \ \ \rho \in T^* \RR^{n-1}\setminus \overline{S}_j\,,
 \quad\liminf_{\rho \to \infty}q_j(\rho )>0.
\end{split}
\end{equation}
Lemma \ref{l:saG} shows that one can choose $ Q_j $ so that
$$ 
Q_j^w ( x' , h D_{x'} ) \; : \; H_{g_j} ( \RR^{n-1} ) 
\longrightarrow H_{g_j} ( \RR^{n-1} ) \quad\text{is self-adjoint}.
$$
Under the assumptions \eqref{gl.2.5}, 
we know that $Q_j$ has discrete spectrum in a
fixed neighbourhood of $\RR_-$ when $h>0$ is
small enough. Let 
\be\label{gl.3}
\cH_j \defeq \Pi_j \big(H_{g_j} ( \RR^{n-1} )\big)\,,
\ \text{ where $\Pi_j \defeq \bbbone_{\RR_-}\big(Q^w_j( x', h D_{x'})\big)$\,, } 
\ee
that is, $\Pi_j$ is the spectral projection
corresponding to the negative spectrum of $Q_j^w$.
In particular, 
\be\label{e:projection}
\|\Pi_j\|_{H_{g_j}\to H_{g_j}}=1,\qquad \dim(\cH_j) \sim c_j\,h^{1-n} \,,\quad c_j>0\,. 
\ee
We group together these projectors in a diagonal matrix $\Pi_h\defeq
{\rm diag}(\Pi_1,\ldots,\Pi_N)$ projecting $
H_{g_1}(\RR^{n-1})\times\cdots H_{g_N}(\RR^{n-1})$ onto $\cH\defeq \cH_1\times\cdots\cH_N$.

The space $\cH_j$ will be equipped with the norm $\|\bullet\|_{H_{g_j}}$.
For future reference we record the following lemma based
on functional calculus of pseudodifferential operators 
(see for instance \cite[Chapter 7]{DiSj}):
\begin{lem}
\label{l:func}
For any uniformly bounded family of states $u=(u(h) \in L^2(\RR^{n-1}))_{h\to 0}$,
\[ \WFh ( u ) \Subset S_j \ 
\Longrightarrow  \| u - \Pi_j u \|_{ H_{g_j} } = \cO(h^\infty) \| u \|_{H_{g_j}} \,. 
\]
\end{lem}
 
\medskip

In \S\ref{ewn} 
we used the microlocally defined operators
\[  
R_+^j(z): H_G( \Omega_j )  \to H_{g_j} ( \tSigma_j) \,. 
\]
Renaming them $R_{+,m}^j(z)$ (where $ m $ stands for {\em microlocal}) we now
define
\be
\label{eq:R+nf} R_+^j(z) \defeq \Pi_j\, R_{+,m}^j: H_G(X) \to \cH_j\,. 
\ee
The estimate \eqref{e:R_+-G} together with the above Lemma shows that
\be
\|R_+^j ( z )\|_{H_{G}(X) \to  \cH_j}= \cO(h^{-M_0\eps}) \,,\qquad z\in \DOCh \,.
\ee
The operators $ R_+^j ( z ) $ are globally well-defined once we
choose a specific realization of $R_{+,m}^j ( z )$, which gives a unique definition mod $\cO(h^\infty )$. 
We have thus obtained a family of operators
$$ 
R_+ ( z )  \defeq (R_+^1,\ldots,R_+^N) \; : \; H_G ( X ) \longrightarrow \cH_1 \times \cdots \cH_N \,.
$$ 
In turn, the operators $R_-^j ( z )$ 
are obtained by selecting a realization of the microlocally defined operator 
$ R_{-,m}^j ( z ) $ on $H_{g_j}(\tSigma_j)$, and restricting that realization
to $\cH_j$:
\begin{equation}
\label{eq:R-nf}
R_{-}^j ( z )=R_{-,m}^j ( z )\,\Pi_j \; : \; \cH_j\longrightarrow  H_G (X)\,.
\end{equation}
Again, these operators are well defined mod $\cO(h^\infty )$. 
Putting together \eqref{e:R_+-G} with \eqref{e:projection}
ensures that
$$
\|R_-^j ( z )\|_{\cH_j\to H_G}=\cO(h^{-M_0\eps})\,.
$$
We group these operators into
\be
\begin{split}
R_- ( z ) \; &:\; \cH_1 \times \cdots 
\cH_N \longrightarrow H_G(X) \\
R_-(z)u_- &= \sum_{j=1}^N R_-^j ( z )\,u_-^j \,, \qquad
u_- = (u_-^1,\ldots,u_-^N) \,. 
\end{split}
\ee

\subsection{A well posed Grushin problem}\label{s:final} 
With these definitions we consider the following Grushin problem:
\begin{gather}
\label{gl.4}
\begin{gathered}
\cP(z) \; : \; 
H_G^2\times\cH \to H_G \times \cH,\qquad  \cH\defeq\cH_1 \times \cdots \cH_N\,, \\
\cP(z) \defeq \begin{pmatrix} ( i/h) ( P_{\theta,R}(h)-z)  &R_-(z)\\ \ \ 
R_+ ( z ) 
    &0 \end{pmatrix}\,,\qquad z\in   \DOCh \,. 
\end{gathered}
\end{gather}
Since $P_{\theta,R}(h)-z$ (which we will denote by $P-z$ for short) is a Fredholm operator,
so is $ \cP ( z ) $, as we have only added finite dimensional spaces.
For $ \Im z > 0 $ the operator $  (P - z) $ is 
invertible, so Lemma~\ref{l:indg} shows that the index of 
$ \cP ( z ) $ is $ 0 $.
Hence, in order to
prove that $\cP(z)$ is bijective it suffices to to construct an approximate
right inverse and then use a Neumann series. 
The rest of this section will be devoted to 
the proof of this (approximate) right invertibility of $\cP(z)$.

\subsubsection{A well-posed homogeneous problem}
As before we first consider the homogeneous problem
\ekv{gl.5}
{
{(i/ h) (P-z)u+R_-(z)u_-=0\,,  \qquad R_+(z)u=v_+\,,}
}
where only one component $v_+^k$ is nonzero (we may assume that $\|v_+^k\|_{\cH_1}=1$).
For that we adapt the methods of \S\ref{s:homog-K0}. We 
construct an approximate solution using the extended
Poisson operator $K^+_k( z)$ (that operator acts on the microlocal space 
$H_{g_k}(\tSigma_k)$, so its action on $\cH_k$ is well-defined modulo $\cO(h^\infty)$),
and take
\[ 
u =\chi _b^k\,K^+_k(z)\,v_+^k \,, 
\]
where $ \chi_b^k $ is the backwards cutoff function with properties
given in \eqref{so.5.0},\eqref{so.4.1} and \eqref{so.5}. The fact that
$G$ increases along the trajectories implies that $u$ satisfies the
same norm bound as with the ``old norms'' (see \eqref{e:uk-norm}):
$$
\| u \|_{H_G(X)}\lesssim h^{-M_0(t_{\max}+\eps)}\,\|v_+^k\|_{\cH_k}\,.
$$
The microlocally defined operator satisfies
$$ 
R^k_{+,m}(z)\,u \equiv v_+^k + \cO_{H_{g_k}}(h^\infty)\,,\qquad R^j_{+,m}(z)\,u =\cO_{H_{g_j}}(h^\infty),\quad j\neq k\,.
$$
As a result, projecting the left hand side onto $\cH^k$ has a negligible effect:
$$
R_{+}^k(z)\,u = \Pi_k(v_+^k+\cO(h^\infty))=v_+^k+\cO_{\cH_k}(h^\infty)\,.
$$
Following \eqref{eq:sol1} we write
\begin{equation}
\label{eq:nnew}
(i/ h) (P-z)u \equiv [(i/h) P,\chi_b^k]\,K^+_k(z) v_+ \in H_G (X) \,. 
\end{equation}
As noticed in \S\ref{s:homog-K0}, the transport properties of $K^+_k(z)$ show that
$u$ is microlocalized inside the union of tubes $\cup_{j\in J_+(k)}T_{jk}^{++}(z)$, so
the right hand side in \eqref{eq:nnew} 
splits into a component concentrated near $\tD_{k}$, and other components concentrated near 
the arrival sets $A_{jk}(z)$, $j\in J_+(k)$. We rewrite \eqref{eq:sol4} for the present data:
\be\label{e:sol4-wp}
(i/ h) (P-z)u \equiv R_{-,m}^k (z) v_+^k - \sum_{j\in J_+(k)}R_{-,m}^j ( z ) \cM_{jk}( z ) v_+^k\,.
\ee
Each state $\cM_{jk}( z ) v_+^k$ is microlocalized inside the 
arrival set $\tA_{jk}(z)\subset\tSigma_j$, which is not contained in $S_j$ in general -- see the remark at the
end of \S \ref{gps} and Fig.~\ref{f:setV}. 

Consequently one could fear that
replacing the operators $R_{-,m}^j(z)$ by the truncated operators
$R_{-}^j(z)$ would drastically modify the above right hand side.
The microlocally weighted spaces $H_G$, $H_{g_j}$ have been constructed precisely to avoid this problem.
The mechanism is a direct consequence of the relative properties of the sets $S_j$ and $V$ explained in
Lemma~\ref{l:t_0}. Namely, a point $\rho_k\in S_{jk}$ is either ``good'', if its image 
$\rho_j=F_{jk,z}(\rho_k)\in S'_j$, or ``bad'', in which case
\begin{equation}
\label{eq:G00}
G_0(\rho_j) - G_0(\rho_k)\geq t_0 \,.
\end{equation}
Let us choose a cutoff
\be\label{e:chi_j}
\chi_{j}\in\CIc(S_j),\quad \chi_j=1\ \ \text{on } S'_j\,,\quad \chi_j=0\ \ \text{outside }\neigh(S'_j,S_j)\,.
\ee
Since the Fourier integral operator $\cM_{jk}(z):H(\tD_{k})\to H(\tA_{jk}(z))$ is uniformly bounded, 
\eqref{eq:G00} implies the norm estimate (see Lemma~\ref{l:G})
$$
\forall v_+^k\in \cH_k\,,\qquad \|(1-\chi^w_j)\,\cM_{jk}(z)\,v_+^k
\|_{H_{g_j}}\lesssim h^{M t_0-M_0 t_{\max}}\,\|v_+^k\|_{\cH_k},\qquad z\in \DOCh\,.
$$
For this estimate to be small when $h\to 0$, we require the
ratio $M_0/M$ to be small enough to ensure the condition
$$
t_0 -\frac{M_0}{M}t_{\max}\geq t_0/2 > 0\,.
$$
(The bounds \eqref{e:theta} and $M_0\leq M_1$ show that the ratio
$M_0/M$ can indeed be chosen arbitrary small.)

On the other hand, $\chi^w_j\,\cM_{jk}(z)\,v_+^k$ is microlocalized inside
$\neigh(S'_j, S_j)$, so Lemma~\ref{l:func} implies that
$(\Pi_j-1)\chi^w_j\,\cM_{jk}(z)\,v_+^k=\cO_{H_{g_j}}(h^\infty)$.
Putting these estimates altogether, we find that
\be
\forall v_+^k\in \cH_k,\qquad \cM_{jk}(z)\,v_+^k=\Pi_j\,\cM_{jk}(z)\,v_+^k + \cO(h^{Mt_0/2})\,\|v_+^k\|\,.
\ee
This crucial estimate shows that the projection of
$\cM_{jk}(z)\,v_+^k$ on $\cH_j$ has a negligible effect. We now define the
finite rank operators
\be\label{e:M_jk}
\tM_{jk}(z)\defeq \begin{cases} 
\Pi_{j}\,\cM_{jk}(z)\,\Pi_{k}:\cH_k\to \cH_j\,,\ j\in J_+(k)\,,\\
0\quad\text{otherwise}\,,
\end{cases}
\text{in short}\ \tM(z)=\Pi_h \cM(z)\,\Pi_h\,.
\ee
This operators satisfy the same norm bounds \eqref{e:M-norm-bound} as their infinite rank counterparts.
Using these operators, and remembering that the operators
$R_{-}^j:\cH_j\to H_G(X)$ are bounded by $\cO(h^{-M_0\eps})$, 
we rewrite \eqref{e:sol4-wp} as
$$
(i/ h) (P-z)u \equiv R_{-}^k (z) v_+^k - \sum_{j\in J_+(k)} R_{-}^j( z
) \tM_{jk}(z)\,v_+^k + \cO(h^{M t_0 / 3})\,\|v_+^k\| \,.
$$
Generalizing the initial data to arbitrary $v_+\in \cH_1\times\cdots\times\cH_N$, we obtain the
\begin{prop}\label{gl1}
Assume $z\in \DOCh$. Let $v_+\in \cH$. Then 
there exists $ (u,u_-) \in H^2_G(X) \times \cH$ such that 
\begin{align}\label{gl.6}
(i/ h) (P-z)u + R_- ( z )\,u_- &= \cO (h^{M t_0 / 3})\| v_+ \|_{\cH}\qquad \hbox{in }H_G(X),\\
\label{gl.6.5}
R_+( z ) u&=v_+ + \cO (h^\infty)\,\| v_+\|_{\cH} \qquad\hbox{in }\cH,\\
\label{gl.7}
\| u\|_{H_G(X)}\lesssim h^{-M_0(t_{\max}+\eps)}\,\|
v_+\|_{\cH},&\qquad \| u_-\|_{\cH} \lesssim h^{-M_0 t_{\max}}\| v_+\|_{\cH}.
\end{align}
The second part of the solution, $u_-$, is of the form
$$
u_-=  ( \tM(z) - Id ) v_+ \,,\quad \| \tM(z)\|_{\cH\to \cH} \lesssim h^{-M_0 t_{\max}}\,,
$$
where $\tM ( z ) = ( \tM_{jk} ( z ) )_{j,k=1,..,N}$ is the matrix of operators defined in \eqref{e:M_jk}. 
\end{prop}

We collect some properties of the operators $\tM_{jk}(z)$, $j\in
J_+(k)$, for $z\in [-\delta,\delta]$:
\begin{itemize}
\item $\tM_{jk}(z)$ is uniformly bounded, and $\mathrm{WF}_h'( \tM_{jk} ( z ))\subset \overline{S}_j\times\overline{S}_k$.
\item take $\rho _k\in \overline{S}_k$, $\rho _j=\tilde{F}_{jk,z}(\rho _k)\in \overline{S}_j$:
\begin{enumerate}
\item if the trajectory segment connecting the points $\kappa_{k,z}(\rho_k)$, $\kappa_{j,z}(\rho_j)$ is contained in $W$, 
   then microlocally near $(\rho _j,\rho _k)$,  $\tM_{jk} ( z )$ is an $h$-Fourier integral operator of order
  zero with associated canonical transformation $\tilde{F}_{jk,z}=\kappa_{j,z}^{-1}\circ F_{jk,z}\circ \kappa_{k,z}$ 
\item if furthermore the above segment is disjoint from the support of
$G$, then $\tM_{jk} ( z )$ is microlocally unitary near $(\rho _j,\rho _k)$.
\item if, on the opposite, this segment contains a part outside $W$, 
  then there exist $\chi _j\in  \CIc (\neigh(\rho _j))$, $\chi _k\in  \CIc (\neigh(\rho _k))$, equal to $1$ near $\rho _j$
  and $\rho _k$ respectively, and a time $t(\rho_k)>0$ independent of
  the exponent $M$, such that 
$$
\chi^w_j\,\tM_{jk} ( z) \chi^w_k=\cO(h^{M\,t(\rho_k)}): H_{g_k}\to H_{g_j}\,.
$$ 
\end{enumerate}
\end{itemize}
For $z\in\DOCh$ similar statements hold, modulo the fact that the
symbol of the Fourier integral operator is multiplied by
$\exp(-izt_+/h)$, which modifies the order of the operator.

\subsubsection{A well-posed inhomogeneous problem}
Let us now consider the inhomogeneous problem
\be\label{e:inhomog2}
(i/h) (P_{\theta,R}-z)u+R_-(z)u_- = v\,\qquad v\in H_G(X)\,.
\ee
We will use a partition of unity to decompose $v$ into several component.

Take $\psi_\delta\in S(T^*X)$, $\psi_\delta=1$ near $p^{-1}([-\delta/2,\delta/2])$,
and $\psi_\delta=0$ outside $p^{-1}([-\delta,\delta])$. The operator $(P_{\theta,R}-z)$ is elliptic outside $p^{-1}[-\delta/2,\delta/2]$.
Taking $\tpsi_\delta$ similar with $\psi_\delta$ but with $ \supp \tpsi_\delta \subset p^{-1} ( [ -\delta/2, \delta/2])$,
the operator 
$$L\defeq (P_{\theta,R}-z-i\tpsi_\delta^w):H^2_G\to H_G$$ 
is invertible, with uniformly bounded inverse $L^{-1}\in\Psi_h^{0}$.
Hence, by taking 
$$
u=(h/i) L^{-1}\,(1-\psi_\delta^w)\,v\,,
$$
we find
$$
(i/h) (P_{\theta,R}-z) u = (i/h) (P_{\theta,R}-z-i\tpsi_\delta^w) u + \cO(h^\infty)\,\|u\|= (1-\psi_\delta^w)\,v+\cO(h^\infty)\|v\|\,,
$$
which solves our problem for the data $(1-\psi_\delta^w)\,v$. The first equality uses pseudodifferential calculus
and the fact that $ \psi_\delta \equiv 1 $ on the support of $ \tpsi_\delta $:
\[ \tpsi_\delta^w L^{-1} ( 1 - \psi_\delta^w) = {\mathcal O}_{{\mathscr S}' \rightarrow {\mathscr S}} ( h^\infty) \,. \]

Let us now consider the data $(\psi_\delta^w v)$ microlocalized in $p^{-1}([-\delta,\delta])$. 
We split this state using a spatial cutoff $\psi_{R}\in \CIc(X)$, such that $\psi_{R}=1$  in
$B(0,R)$, $\psi_R=0$ outside $B(0,2R)$. To solve the equation 
\be\label{e:inhom33}
(i/h) (P_{\theta,R}-z)u= \wtv,\qquad \wtv=(1-\psi_R)\, \psi_\delta^w\,v\,, 
\ee
we take the Ansatz
\be\label{e:simple}
u= E(z)\,\wtv\,,
\ee
with $E(z)$ the parametrix of \eqref{ml.9} (with $P$ replaced by
$P_{\theta,R}$), for the same time $T=t_{\max}+\eps$ as in \eqref{eq:T}. It satisfies 
\be\label{e:simple2}
(i/h)(P_{\theta,R}-z)u=\wtv - e^{-iT(P_{\theta,R}-z)/h}\,\wtv\,.
\ee
The time $T$ is small enough, so that 
\[  
\Phi^t \left( p^{-1}([-\delta,\delta])\setminus T^*_{B(0,R)}X \right) \cap T^*_{B(0,R/2)}X = \emptyset 
\,, \ \ 0 \leq t \leq T\,. 
\]
Hence, the states 
\[ 
\wtv(t) \defeq e^{-it(P_{\theta,R}-z)/h}\wtv 
\]
are all microlocalized outside $T^*_{B(0,R/2)}X$
for $t\in [0,T]$. 
The estimate \eqref{eq:estaway} (adapted to the weight $G_0$) then implies that \cite[Lemma 6.4]{NZ2}
$$
\partial_t \|\wtv(t)\|^2_{H_G}=\frac{2}{h}\Im \la
(P_{\theta,R}-z)\wtv(t),\wtv(t)\ra_{H_G} \leq (-M_1/C_1 + 2M_0)\,\log(1/h)\,,\quad \forall t\in [0,T]\,,
$$
where $C_1>0$ is independent of the choice of $M_1$. Once more,
we assume $M_0/M_1$ is small enough so that $-M_1/C_1 + 2M_0\leq -
M_1/2C_1$, and hence
$$
\|e^{-iT(P_{\theta,R}-z)/h}\,\wtv\|_{H_G}\leq C\,h^{M_1 T /2C_1}\|\wtv\|_{H_G}\,, 
$$
so the problem \eqref{e:inhom33} is solved modulo a remainder
$\cO(h^{M_1 T /2C_1})$.

We now consider the component $(\psi_R\psi^w_\delta v)$ microlocalized
in $T^*_{B(0,2R )}\cap p^{-1}([-\delta,\delta])$. 
We split it again using
a cutoff $\psi_{V_1}\in\CIc(V_1)$, $\psi_{V_1}=1$ in the set
$V'_1\Subset V_1$ (see the discussion after Lemma~\ref{l:G0}). 
To solve the problem for the inhomogeneous data 
$$
\wtv=(1-\psi^w_{V_1})\psi_R\psi^w_\delta v\,,
$$ 
we use the Ansatz \eqref{e:simple},
resulting in the estimate \eqref{e:simple2}. The microlocalization of $\wtv$ outside of $V'_1$, 
together with the assumption \eqref{e:V-V_1}, implies the norm estimate (see Lemma~\ref{l:G}) 
$$
\|e^{-iT(P_{\theta,R}-z)/h}\,\wtv\|_{H_G}\leq C\,h^{ M t_1 / 2 - M_0 T }\|\wtv\|_{H_G}\,.
$$
Again, we assume $M_0/M$ small enough, so that $M t_1 / 2 - M_0 T \geq
M t_1 / 3 $. We have solved the problem for $\wtv$ up to a remainder
$\cO(h^{M t_1 / 3}) \|\wtv\|_{H_G}$.

We finally consider the data $\wtv=\psi^w_{V_1}\psi_R\psi^w_\delta v$
microlocalized inside $V_1$. For this data, we can use the microlocal
analysis of \S\ref{s:inhomog-K0}. If $\WFh(\wtv)$ is contained inside $V_1\cap\whT^{--}_{jk}$, then
$ \WFh (\chi _b^k\,E(z)\,\wtv ) $ (see the Ansatz \eqref{eq:uu}) 
will intersect $\Sigma_j$ inside the arrival set $\tA_{jk}(z)$, but not
necessarily inside $S_{j}$. However, the same phenomenon as in Lemma~\ref{l:t_0} occurs:
there exists a time $t_3>0$ such that, for any $z\in [-\delta,\delta]$ and any  $\rho\in V_1\cap T^{--}_{jk}(z)$, 
\be\label{e:t_3}
\rho_+(\rho)\in \Sigma_j(z)\setminus \kappa_{j,z}(S'_j)
\Longrightarrow G_0(\rho_+(\rho)) - G_0(\rho)\geq t_3\,.
\ee
If we decompose $R_{+,m}^j( z )E(z)\wtv$ using the cutoff $\chi_{j}$ of \eqref{e:chi_j},
the property \eqref{e:t_3} implies that
$$
\|(1-\chi_j^w)\,R_{+,m}^j( z )E(z)\wtv\|_{H_{g_j}}=\cO(h^{M t_3 / 2 - M_0 T})\|\wtv\|_{H_G}\,.
$$
Again we assume $M_0/M$ small enough, so that $M t_3 / 2 - M_0 T \geq
M t_3 / 3$.
Hence, if we set
\begin{align*}
u_-^j&= R_+^j ( z ) \chi_j^w\,E(z)\wtv\\
&=R_{+,m}^j( z ) \chi_j^w E(z)\wtv+\cO(h^\infty) \|\wtv\|_{H_G} \\
&= R_{+,m}^j( z ) E(z)\wtv+\cO(h^{Mt_3/3}) \|\wtv\|_{H_G}\,,
\end{align*}
we end up with a solution of \eqref{e:inhomog2} modulo a remainder
$\cO(h^{ M t_3 / 3 })\|\wtv\|_{H_G}$.


\bigskip

We recall that $M_1/M$ is bounded by \eqref{e:theta}, so all the above
error estimates can be put in the form $\cO(h^{cM}) \|\wtv\|_{H_G}$, with $c>0$
independent of $M$: we have thus
shown that the problem \eqref{e:inhomog2} admits a solution for any
$v\in H_G$, up to this remainder. We may then
apply Proposition \ref{gl1} to solve the resulting homogeneous
problem, and get an approximate solution for the full problem \eqref{gl.4}. We 
summarize this solution in the following
\begin{prop}\label{gl3}
Assume $z\in \DOCh $. Let $(v,v_+)\in H_G\times \cH$. Then
there exists $(u,u_-)\in H^2_G\times \cH$ such that 
\begin{align}\label{gl.10}
&\begin{cases} (i/h)(P-z) u + R_-(z) u_-&= v+\cO (h^{cM})(\Vert v\Vert_{H_G}+ \Vert v_+\Vert_{\cH})\qquad \hbox{in }H_G(X)\,, \\  
 R_+( z ) u&=v_+ +\cO(h^{\infty})\ (\Vert v\Vert_{H_G} + \Vert v_+\Vert_{\cH})\qquad\hbox{in }\cH\,,\end{cases}\\
\label{gl.11}&\Vert u\Vert_{H_G^2}+\Vert u_-\Vert_{\cH} \lesssim
h^{-M_0 (2 t_{\max} + 2\eps)}\,
\big( \Vert v\Vert_{H_G}+ \Vert v_+\Vert_{\cH} \big)\,.
\end{align}
\end{prop}

\subsection{Invertibility of the Grushin problem}
\label{igp}

We can transform this approximate solution into an exact one. The system \eqref{gl.10} can be expressed as an approximate
inverse of $\cP(z)$:
\begin{gather}
\label{eq:hooarow}
\begin{gathered}
\binom{u}{u_-}=\widetilde\cE(z)\binom{v}{v_+},\\
\cP(z)\,\widetilde\cE(z)=
I+\cR(h):H_G\times\cH \longrightarrow  H_G\times\cH \,,\quad \| \cR(h) \|=\cO(h^{cM})\,.
\end{gathered}
\end{gather}
For $h$ small enough the operator $I+\cR(h)$ can be inverted by a 
Neumann series, so we obtain an exact right inverse of $\cP(z)$,
$$
\cE(z)=\widetilde\cE(z)\,(I+\cR(z))^{-1}\,.
$$  
Since $\cP(z)$ is of index zero, $\cE(z)$ is also a left inverse, which proves the well-posedness
of our Grushin problem \eqref{gl.4}.
\begin{thm}
\label{t:mg}
We consider $h>0$ small enough, and $z\in \DOCh$. 
For every $(v,v_+)\in H_G\times \cH$, there exists a unique
$(u,u_-)\in H_G^2 \times \cH$ such that 
\begin{equation}
\label{gl.12}
\begin{cases}
(i/ h) (P_{\theta,R}-z)u+R_-( z ) u_- & = v \\
R_+ ( z ) u & =v_+ \,, \end{cases}
\end{equation}
where $ R_\pm ( z ) $ are defined by \eqref{eq:R+nf} and \eqref{eq:R-nf}.
The estimates (\ref{gl.11}) hold, so if we write 
$$
\binom{u}{u_-}=\cE(z)\,\binom{v}{v_-},\qquad \cE(z)=\begin{pmatrix}E &E_+\\ E_- &E_{-+} \end{pmatrix}\,,
$$
then the following operator norms (between the appropriate Hilbert
spaces) are bounded by:
\ekv{gl.13}
{
\Vert E\Vert\,,\ \   \Vert E_+\Vert,\ \ 
\Vert E_-\Vert \,, \ \  \Vert E_{-+}\Vert = \cO ( h^{-M_0 (2 t_{\max} + 2\eps)} ).
}
Moreover, we have a precise expression for the effective Hamiltonian:
\ekv{gl.14}
{
E_{-+}(z)= -I + \tM ( z ) +\cO_{\cH\to\cH}(h^{c'M})\defeq - I + M(z,h)\,,
}
where $\tM(z)$ is the matrix of ``open quantum maps'' defined in
\eqref{e:M_jk} and described after Proposition \ref{gl1}.
\end{thm}
\noindent{\bf Remark.} 
If we restrict the parameter $z$ to a rectangle of height $|\Im z|\leq Ch$ instead
of $|\Im z|\leq M_0 h\log(1/h)$, the bounds \eqref{gl.14} become $\|E_*(z)\|=\cO(1)$.

\smallskip

Theorem~\ref{t:s} and the formula \eqref{eq:mumu} follow from this more precise result. In fact, 
the equality \eqref{eq:two_r} shows that 
\begin{equation}
\label{eq:rank}
\rank \oint_z \chi R ( w ) \chi\, d w = \rank \oint_z \chi R_{\theta,R} (w ) \chi\, d w = 
- \frac 1 { 2\pi i } \tr \oint_z R_{\theta,R} ( w ) \, d w\,,
\end{equation}
see \cite[Proposition 3.6]{SjZw91} for the proof of the last identity 
in the simpler case of compactly supported perturbations, and 
\cite[Section 5]{Sj} for the general case.

The well-posedness of our Grushin problem means 
that we can apply formula \eqref{eq:trace} recalled in \S \ref{gp}.
It shows that 
the right hand side in \eqref{eq:rank} is
equal to 
$$
\frac 1 { 2\pi i } \tr \oint_z E_{-+}(w) ^{-1} 
E_{-+}'(w)\, d w\,,
$$
which in view of \eqref{gl.14} gives \eqref{eq:mumu}.
The exponent $L\defeq c'M$ in the remainder of \eqref{gl.14} depends on the integer $M>0$ used in the
scaling of the weight function $G$, which can be chosen arbitrary
large, independently of $c'>0$.


\end{document}